\title{Topological expansion for Haar-distributed orthogonal matrices and second-order freeness of orthogonally invariant ensembles}
\author{C.\ E.\ I.\ Redelmeier}
\date{March 8, 2013}

\documentclass[11pt]{article}

\usepackage{graphicx}
\usepackage{graphics}
\usepackage{amsfonts}
\usepackage{amssymb}
\usepackage{amsthm}
\usepackage{amsmath}
\usepackage{color}
\usepackage{cite}
\usepackage{bbm}

\newtheorem{theorem}{Theorem}[section]
\newtheorem{lemma}[theorem]{Lemma}
\newtheorem{proposition}[theorem]{Proposition}
\newtheorem{corollary}[theorem]{Corollary}

\theoremstyle{remark}

\newtheorem{example}[theorem]{Example}

\theoremstyle{definition}
\newtheorem{definition}[theorem]{Definition}

\begin{document}

\maketitle

\begin{abstract}
We present a genus expansion-type expression for the expected values of products of traces of expressions involving Haar-distributed orthogonal matrices.  As with other real genus expansions, nonorientable surfaces appear, in addition to the orientable surfaces of the complex expansion.

We use this expression to demonstrate that independent random matrices which are orthogonally in general position, such as matrices whose distributions are orthogonally invariant, are asymptotically real second-order free.
\end{abstract}

\section{Introduction}

The connection between free probability and random matrices was first described in \cite{MR1094052}.  Many important ensembles of random matrices, including any independent matrices with unitarily or orthogonally invariant distributions, are asymptotically free.

This connection may be extended to the second-order statistics, or fluctuations, of matrices.  Second-order freeness is defined in \cite{MR2216446}, and it is shown that independent matrices which are unitarily in general position are second-order free in \cite{MR2294222}.  (See also \cite{MR2302524}.)  However, the real analogues of the complex matrix models do not generally obey the definition of complex second-order freeness.  A definition for real second-order freeness, satisfied by real Ginibre matrices, Gaussian orthogonal matrices, and real Wishart matrices, is given in \cite{2011arXiv1101.0422R}.

Expected values of traces of many complex random matrices may be calculated as sums over orientable surfaces, in which the order of the term depends only on the genus of the surface (see \cite{MR2036721}, Chapter~3, for details of this construction for Gaussian unitary matrices, and \cite{MR2052516} for complex Wishart matrices).  In the analogous expressions for real random matrices, nonorientable surfaces appear in addition to the orientable ones (see \cite{2011arXiv1101.0422R, 2012arXiv1204.6211R} for expansions for details of the expansions for real Ginibre, Gaussian orthogonal, and real Wishart matrices).  We present a similar expression for Haar-distributed orthogonal matrices, similarly analogous to the expression for Haar-distributed unitary matrices in \cite{MR2294222}.  While this expression resembles those for the other real matrix models, it lacks the property that connected components contribute multiplicatively, so many of the approaches of \cite{2011arXiv1101.0422R} cannot be used.  However, a cumulant approach allows us to consider the contribution of connected components for large matrices.  This diagrammatic representation gives us expressions for the first- and second-order statistics of independent matrices which are orthogonally in general position, demonstrating that they are asymptotically real second-order free.

In Section~\ref{section: notation}, we give the notation, constructions, and lemmas we will be using, including summaries of the cartographic machinery (see \cite{MR0404045, MR2036721}, and \cite{MR1813436, MR2851244, 2011arXiv1101.0422R, 2012arXiv1204.6211R} for the nonorientable case), the annular noncrossing conditions from \cite{MR2052516}, and the Weingarten calculus for Haar-distributed orthogonal matrices from \cite{MR1959915, MR2217291, MR2567222}.  In Section~\ref{section: genus expansion}, we present an exact formula for the expected value of the product of traces of expressions involving Haar-distributed orthogonal matrices and outline how asymptotic values may be computed diagrammatically.  We give an example calculations with diagrams.  In Section~\ref{section: cumulants}, we give exact and asymptotic expressions for the cumulants of traces.  In Section~\ref{section: freeness}, we show that matrices which are orthogonally in general position are real second-order free.  In particular, this demonstrates that independent random matrices whose distributions are orthogonally invariant are real second-order free.

\section{Notation}
\label{section: notation}

For an integer \(n>0\), we denote the set of integers \(\left\{1,\ldots,n\right\}\) by \(\left[n\right]\).  For any set \(I\subseteq\mathbb{Z}\), we let \(-I:=\left\{-k:k\in I\right\}\), and we let \(\pm I=I\cup\left(-I\right)\).

We will let \(\delta:k\mapsto -k\).  For an even \(\varepsilon:I\rightarrow\left\{1,-1\right\}\) (or one which we may extend to an even function on \(\pm I\)), we let \(\delta_{\varepsilon}:k\mapsto\left(-1\right)^{k}k\).

\subsection{Partitions}

\begin{definition}
A {\em partition} of a set \(I\) is a set of subsets \(V_{1},\ldots,V_{n}\subseteq I\) (which we call {\em blocks}) such that \(V_{k}\neq\emptyset\), \(V_{k}\cap V_{l}\) for \(k\neq l\), and \(V_{1}\cup\cdots\cup V_{n}=I\), \(1\leq k,l\leq n\).  We denote the number of blocks in a partition \(\pi\) by \(\#\left(\pi\right)\).  We will denote the set of all partitions of \(I\) by \({\cal P}\left(I\right)\), and the set of all partitions of \(\left[n\right]\) by \({\cal P}\left(n\right)\).

The partitions of \(I\) form a partially-ordered set: if \(\pi,\rho\in{\cal P}\left(I\right)\), then we say that \(\pi\preceq\rho\) (\(\pi\) is finer than \(\rho\) or \(\rho\) is coarser than \(\pi\)) if every block of \(\pi\) is contained in a block of \(\rho\).  In fact, \({\cal P}\left(I\right)\) is a lattice: for \(\pi,\rho\in{\cal P}\left(I\right)\), we define the join of \(\pi\) and \(\rho\), \(\pi\vee\rho\), as the smallest partition larger than both \(\pi\) and \(\rho\), and the meet \(\pi\wedge\rho\) as the largest partition smaller than both \(\pi\) and \(\rho\).  We denote the smallest element of \({\cal P}\left(I\right)\) (in which all elements are singlets) by \(0_{I}\), and the smallest element of \(\left[n\right]\) by \(0_{n}\).  We denote the largest element (which consists of a single block \(I\)) by \(1_{I}\), and the largest element of \(\left[n\right]\) by \(1_{n}\).

If blocks \(V_{1},V_{2}\in\pi\) are subsets of the same block of \(\pi\vee\rho\), we say that \(\rho\) connects these blocks of \(\pi\).

If \(f\) is a function with domain \(I\), we define \(\ker\left(f\right)\in{\cal P}\left(I\right)\) as the partition whose blocks are the preimages of points in the range of \(f\).

A useful lemma:
\begin{lemma}
\label{lemma: useful}
If \(\pi\preceq\rho\), \(\pi,\rho\in{\cal P}\left(I\right)\) for some set \(I\), then
\[\#\left(\pi\right)-\#\left(\pi\vee\sigma\right)\geq\#\left(\rho\right)-\#\left(\rho\vee\sigma\right).\]
\end{lemma}
\begin{proof}
Since each partition considered is coarser than \(\pi\), we can consider each partition to be an element of the partitions on the blocks of \(\pi\), \({\cal P}\left(\pi\right)\), where a block corresponds to the set of blocks of \(\pi\) contained in it.  The number of blocks in a partition is thus the same in both posets.

A block \(V\) of \(\sigma\) may join at most \(\left|V\right|\) blocks of \(\rho\), so taking the join of \(\rho\) with \(\sigma\) reduces the number of blocks by at most \(\sum_{V\in\sigma}\left(\left|V\right|-1\right)=\left|I\right|-\#\left(\sigma\right)\).  Considered as elements of \({\cal P}\left(\pi\right)\), the number of elements in the underlying set is instead \(\#\left(\pi\right)\), so in fact this join reduces the number of blocks by at most \(\#\left(\pi\right)-\#\left(\pi\vee\sigma\right)\), as desired.
\end{proof}

A partition whose blocks all contain exactly two elements is called a {\em pairing}.  We denote the set of pairings on \(I\) by \({\cal P}_{2}\left(I\right)\) and the set of pairings on \(\left[n\right]\) by \({\cal P}_{2}\left(n\right)\).  We note that this set is empty for \(n\) odd, so any sum over this set is zero.
\end{definition}

\begin{definition}
We define the M\"{o}bius function \(\mu:{\cal P}\left(n\right)^{2}\rightarrow\mathbb{C}\) as the unique function such that \(\mu\left(\pi,\rho\right)=0\) unless \(\pi\preceq\rho\), and for any \(\pi,\rho\in{\cal P}\left(n\right)\),
\[\sum_{\pi\preceq\sigma\preceq\rho}\mu\left(\sigma,\rho\right)=\left\{\begin{array}{ll}1,&\pi=\rho\\0,&\textrm{otherwise}\end{array}\right.\]
\end{definition}

The value of \(\mu\left(\pi,\rho\right)\) depends only on the number of blocks of \(\pi\) contained in each block of \(\rho\).  See, e.g., \cite{MR1311922}, Chapter~12, for more detail on M\"{o}bius functions, and the exercises to Chapter~10 of \cite{MR2266879} for the M\"{o}bius function of partitions in particular.

\begin{definition}
We define the \(n\)th mixed moment \(a_{n}\) of random variables \(X_{1},\ldots,X_{n}\) by
\[a_{n}\left(X_{1},\ldots,X_{n}\right)=\mathbb{E}\left(X_{1}\cdots X_{n}\right)\textrm{.}\]
For \(\pi\in{\cal P}\left(n\right)\), we define the moment \(a_{\pi}\) by
\[a_{\pi}\left(X_{1},\ldots,X_{n}\right)=\prod_{V=\left\{i_{1},\ldots,i_{m}\right\}\in\pi}a_{m}\left(X_{i_{1}},\ldots,X_{i_{m}}\right)\textrm{.}\]

We define the cumulants \(k_{1},k_{2},\ldots\) such that \(k_{n}\) is an \(n\)-linear function, and such that for \(\pi\in{\cal P}\left(n\right)\),
\[k_{\pi}\left(X_{1},\ldots,X_{n}\right):=\prod_{V=\left\{i_{1},\ldots,i_{m}\right\}}k_{m}\left(X_{i_{1}},\ldots,X_{i_{m}}\right)\textrm{,}\]
to be the unique functions satisfying, for all \(n\),
\[a_{n}\left(X_{1},\ldots,X_{n}\right)=\sum_{\pi\in{\cal P}\left(n\right)}k_{\pi}\left(X_{1},\ldots,X_{n}\right)\textrm{.}\]

Equivalently, the cumulants are
\[k_{\pi}\left(X_{1},\ldots,X_{n}\right)=\sum_{\rho\preceq\pi}\mu\left(\rho,\pi\right)a_{\rho}\left(X_{1},\ldots,X_{n}\right)\textrm{.}\]
\end{definition}

We note that the first cumulant \(k_{1}\left(X\right)=\mathbb{E}\left(X\right)\) is the expectation and the second cumulant \(k_{2}\left(X,Y\right)=\mathbb{E}\left(XY\right)-\mathbb{E}\left(X\right)\mathbb{E}\left(Y\right)\) is the covariance.

It is a standard result that a cumulant containing any independent random variables vanishes.

In order to distinguish the partition of a set from the partition of an integer (an unordered list of integers summing to that integer), we will always refer to the partition of an integer as a Young diagram.

\begin{definition}
A {\em Young diagram} on integer \(n\geq 0\) is an unordered list of integers summing to \(n\).  If \(\lambda\) is a Young diagram on \(n\), we write \(\lambda\vdash n\).
\end{definition}

We will refer to the integers in the list \(\lambda\) as the lengths of its rows.  A Young diagram with rows of length \(\lambda_{1}\geq\cdots\geq \lambda_{r}\) may be written \(\left(\lambda_{1},\ldots,\lambda_{r}\right)\).

\subsection{Permutations}

We denote the set of permutations on a set \(I\) by \(S\left(I\right)\), and the set of permutations on \(\left[n\right]\) by \(S_{n}\).  We will generally write permutations \(\pi\) in cycle notation.

The elements in a cycle form an orbit of that partition.  We will use the term orbit to refer to this subset of the domain, and cycle to refer to the action of the permutation on this subset.

Since the orbits of a permutation form a partition of its domain, we will often use the permutation to represent this partition.  A pairing \(\pi\in{\cal P}_{2}\left(I\right)\) corresponds to a permutation in \(S\left(I\right)\), in which each element is mapped to the other element in its block.

If \(\left|\pi\right|\) is the smallest number of factors in an expression of \(\pi\) as a product of transpositions, then
\[\#\left(\pi\right)=\left|I\right|-\left|\pi\right|\textrm{.}\]
This number depends on the implied domain of \(\pi\), which we will state explicitly if it is not clear from context.

We note that if we conjugate \(\pi\in S\left(I\right)\) by \(\rho\in S\left(i\right)\), the cycle notation of \(\rho\pi\rho^{-1}\) is that of \(\pi\), but with each \(k\in I\) replaced by \(\rho\left(k\right)\).

\begin{definition}
If \(\pi\in S\left(I\right)\), then we define the permutation induced by \(\pi\) on \(J\subseteq I\), denoted \(\left.\pi\right|_{J}\), by letting \(\left.\pi\right|_{J}\left(k\right)=\pi^{m}\left(k\right)\) where \(m>0\) is the smallest integer such that \(\pi^{m}\left(k\right)\in J\).
\end{definition}
In cycle notation, this amounts to deleting the elements not in \(J\).  If \(\pi\) does not connect \(J\) and \(I\setminus J\), then \(\left.\pi\right|_{J}\) is just the restriction of \(\pi\) to \(J\), which is a permutation on \(J\).

\subsection{Premaps}

Pairs of permutations may be used to represent orientable maps (or surfaced graphs, or surface gluings), see \cite{MR0404045, MR2036721}.  In order to represent unoriented maps (as appear in the real matrix calculations, see \cite{MR2036721}, Chapter~3), we use what we will call premaps.  (These are not exactly what are referred to as premaps in the existing literature, but are a component of them.  See \cite{MR1813436}.)  For proofs of the cited results on premaps, see \cite{2011arXiv1101.0422R}, and for diagrams motivating the constructions, see \cite{2012arXiv1204.6211R}.

\begin{definition}
A {\em premap} is a permutation \(\pi\in S\left(\pm I\right)\) such that \(\delta\pi\delta=\pi^{-1}\), and such that there is no cycle containing both \(k\) and \(-k\) for \(k\in I\).  We denote the set of premaps on \(\pm I\) by \(PM\left(\pm I\right)\).
\end{definition}
The second condition may be verified by checking that there is no \(k\in I\) such that \(\pi\left(k\right)=-k\) (see \cite{2011arXiv1101.0422R}).

For any \(\pi\in PM\left(\pm\left[n\right]\right)\), \(\pi^{-1}\in PM\left(\pm\left[n\right]\right)\), and for any \(\varepsilon:\left[n\right]\rightarrow\left\{1,-1\right\}\) \(\delta_{\varepsilon}\pi\delta_{\varepsilon}\) is a premap with the same number of cycles, since conjugation by \(\delta_{\varepsilon}\) swaps some \(k\) and \(-k\) in \(\pi\).

\begin{definition}
We call a cycle in \(\pi\in PM\left(\pm I\right)\) {\em particular} if its \(k\) with smallest \(\left|k\right|\) is positive.  We note that for every particular cycle in \(\pi\), we have a cycle that is not particular containing the negatives of the elements (in reverse order).  We will denote by \(\pi/2\) the set of elements of \(\pm I\) contained in particular cycles, as well as the permutation \(\pi\) restricted to this set (so \(\pi/2\) uniquely determines \(\pi\)).
\end{definition}

\begin{definition}
For \(I\) with \(I\cap\left(-I\right)=\emptyset\) and \(\varphi\in S\left(I\right)\), let \(\varphi_{+}:=\varphi\) and \(\varphi_{-}:=\delta\varphi\delta\).  For \(\alpha\in PM\left(\pm I\right)\), define
\[K\left(\varphi,\alpha\right)=\varphi_{+}^{-1}\alpha^{-1}\varphi_{-}\textrm{.}\]
\end{definition}
If \(\varphi\) encodes the face information of a map and \(\alpha\) the edge or hyperedge information, \(K\left(\varphi,\alpha\right)\) may be thought of as encoding the vertex information, but it is its inverse which appears more often in our calculations; see \cite{2012arXiv1204.6211R}.)  As shown in \cite{2011arXiv1101.0422R}, \(K\left(\varphi,\alpha\right),K\left(\varphi,\alpha\right)^{-1}\in PM\left(\pm I\right)\).

\begin{definition}
For \(I\) with \(I\cap\left(-I\right)=\emptyset\), a \(\varphi\in S\left(I\right)\), and a \(\alpha\in PM\left(\pm I\right)\), we define the Euler characteristic by
\[\chi\left(\varphi,\alpha\right):=\#\left(\varphi_{+}\varphi_{-}^{-1}\right)/2+\#\left(\alpha\right)/2+\#K\left(\varphi,\alpha\right)/2-\left|I\right|\textrm{.}\]
\end{definition}

We use several properties of the Euler characteristic (see \cite{2011arXiv1101.0422R} for proofs):
\begin{lemma}
\label{lemma: Euler characteristic}
If \(\varphi\in S\left(I\right)\) has orbits \(V_{1},\ldots,V_{r}\), and if \(\alpha\in PM\left(\pm I\right)\) connects the blocks \(\pm V_{1},\ldots,\pm V_{r}\) (that is, \(\left\{\pm V_{1},\ldots,\pm V_{r}\right\}\vee\alpha=1_{\pm I}\)), then \(\chi\left(\varphi,\alpha\right)\leq 2\).

We note that \(K\left(\varphi,\alpha\right)\preceq\left\{\pm V_{1},\ldots,\pm V_{r}\right\}\vee\alpha\).  Let \(\pm I_{1}\cap\pm I_{2}=\emptyset\) with \(I_{i}\cap\left(-I_{i}\right)=\emptyset\), \(\varphi_{i}\in S\left(I_{i}\right)\), \(\alpha_{i}\in PM\left(\pm I_{i}\right)\), \(i=1,2\), and let \(\varphi=\phi_{1}\phi_{2}\in S\left(I_{1}\cup I_{2}\right)\) and \(\alpha=\alpha_{1}\alpha_{2}\in PM\left(\pm\left(I_{1}\cup I_{2}\right)\right)\).  Then \(\chi\left(\varphi,\alpha\right)=\chi\left(\phi_{1},\alpha_{1}\right)+\chi\left(\phi_{2},\alpha_{2}\right)\).
\end{lemma}

\subsection{Noncrossing Conditions}

Highest order terms are often represented by highest Euler characteristic terms, which are typically those which can be represented as noncrossing diagrams.  Noncrossing conditions may be defined in terms of permutations.  We cite the definitions and characterisations for \(\varphi\) with one and two cycles, for oriented and unoriented surfaces.  (Our conventions are slightly different from those in \cite{MR1475837, MR2052516}, where \(\alpha\) follows the sense of \(\varphi\).  Here we think of each object as being oriented counter-clockwise relative to itself.  See \cite{2012arXiv1204.6211R} for more exposition.)

\begin{definition}
Let \(I\) be a finite set, let \(\varphi\in S\left(I\right)\) be a permutation with a single cycle, and let \(\alpha\in S\left(I\right)\).

We call \(\alpha\) {\em disc nonstandard (relative to \(\varphi\))} if there are three distinct elements \(a,b,c\in I\) such that \(\left.\varphi\right|_{\left\{a,b,c\right\}}=\left(a,b,c\right)\) and \(\left.\alpha\right|_{\left\{a,b,c\right\}}=\left(a,b,c\right)\).  We call \(\alpha\) {\em disc standard (relative to \(\varphi\)} if there are no such elements.

We call \(\alpha\) {\em disc crossing (relative to \(\varphi\))} if there are four distinct elements \(a,b,c,d\in I\) such that \(\left.\varphi\right|_{\left\{a,b,c,d\right\}}=\left(a,b,c,d\right)\) but \(\left.\alpha\right|_{\left\{a,b,c,d\right\}}=\left(a,c\right)\left(b,d\right)\).  We call \(\alpha\) {\em disc noncrossing (relative to \(\varphi\))} if it is neither disc nonstandard nor disc noncrossing.

We denote the set of disc-noncrossing permutations on \(I\) relative to \(\varphi\) by \(S_{\mathrm{disc-nc}}\left(\varphi\right)\).
\end{definition}

The following theorem from \cite{MR1475837} shows that the noncrossing conditions are equivalent to a condition on the number of cycles.
\begin{theorem}[Biane]
Let \(\varphi,\alpha\in S\left(I\right)\) for some finite set \(I\), and let \(\varphi\) have a single cycle.  Then
\[\#\left(\alpha\right)+\#\left(\varphi^{-1}\alpha^{-1}\right)=\left|I\right|+1\]
if and only if \(\alpha\in S_{\mathrm{disc-nc}}\left(\varphi\right)\).
\end{theorem}

Similar conditions for \(\varphi\) with two cycles can be found in \cite{MR2052516}:
\begin{definition}
Let \(\varphi\in S\left(I\right)\) be a permutation with two cycles (which we will refer to as \(\varphi_{\mathrm{ext}}\) and \(\varphi_{\mathrm{int}}\)), and let \(\alpha\in S\left(I\right)\).

We say that \(\alpha\) is {\em annular nonstandard (relative to \(\varphi\))} if any of the following conditions holds:
\begin{enumerate}
	\item there are \(a,b,c\in I\) such that \(\left.\varphi\right|_{\left\{a,b,c\right\}}=\left(a,b,c\right)\) and \(\left.\alpha\right|_{\left\{a,b,c\right\}}=\left(a,b,c\right)\),
	\item there are \(a,b,c,d\in I\) such that \(\left.\varphi\right|_{\left\{a,b,c,d\right\}}=\left(a,b\right)\left(c,d\right)\) but \(\left.\alpha\right|_{\left\{a,b,c,d\right\}}=\left(a,c,b,d\right)\).
\end{enumerate}

We call \(\alpha\) {\em annular standard (relative to \(\varphi\))} if none of these conditions hold.

Let \(x\in\varphi_{\mathrm{ext}}\) and \(y\in\varphi_{\mathrm{int}}\).  We define a permutation \(\lambda_{x,y}\) on \(I\setminus\left\{x,y\right\}\) by letting \(\lambda_{x,y}\left(\varphi^{-1}\left(x\right)\right)=\varphi\left(y\right)\) and \(\lambda_{x,y}\left(\varphi^{-1}\left(y\right)\right)=\varphi\left(x\right)\) (we will generally assume that \(\varphi_{\mathrm{ext}}\) and \(\varphi_{\mathrm{int}}\) have at least two elements), and letting \(\lambda_{x,y}\left(a\right)=\varphi\left(a\right)\) otherwise.  We say that \(\alpha\) is {\em annular crossing (relative to \(\varphi\))} if any of the three following conditions holds:
\begin{enumerate}
	\item there are elements \(a,b,c,d\in I\) such that \(\left.\varphi\right|_{\left\{a,b,c,d\right\}}=\left(a,b,c,d\right)\) but \(\left.\alpha\right|_{\left\{a,b,c,d\right\}}=\left(a,c\right)\left(b,c\right)\),
	\item there are elements \(a,b,c,x,y\in I\), \(x\in\varphi_{\mathrm{ext}}\) and \(y\in\varphi_{\mathrm{int}}\), such that \(\left.\lambda_{x,y}\right|_{\left\{a,b,c\right\}}=\left(a,b,c\right)\) and \(\left.\alpha\right|_{\left\{a,b,c,x,y\right\}}=\left(a,b,c\right)\left(x,y\right)\),
	\item there are elements \(a,b,c,d,x,y\in I\), \(x\in\varphi_{\mathrm{ext}}\) and \(y\in\varphi_{\mathrm{int}}\), such that \(\left.\lambda_{x,y}\right|_{\left\{a,b,c,d\right\}}=\left(a,b,c,d\right)\) but \(\left.\alpha\right|_{\left\{a,b,c,d,x,y\right\}}=\left(a,c\right)\left(b,d\right)\left(x,y\right)\).
\end{enumerate}
We call \(\alpha\) {\em annular noncrossing (relative to \(\varphi\))} if it is neither annular nonstandard or annular noncrossing.
	We denote the set of connected (i.e. those connecting the cycles of \(\varphi\)) annular-noncrossing permutations relative to \(\varphi\) by \(S_{\mathrm{ann-nc}}\left(\varphi\right)\).
\end{definition}

\begin{theorem}[Mingo, Nica]
Let \(\varphi,\alpha\in S\left(I\right)\) for some finite set \(I\), where \(\varphi\) has two cycles and \(\alpha\) connects the cycles of \(\varphi\).  Then
\[\#\left(\alpha\right)+\#\left(\varphi^{-1}\alpha^{-1}\right)=\left|I\right|\]
if and only if \(\alpha\in S_{\mathrm{ann-nc}}\left(\varphi\right)\).
\end{theorem}

The analogous results for premaps (representing diagrams where twists are possible) are shown in \cite{2011arXiv1101.0422R}:
\begin{theorem}
\label{theorem: unoriented noncrossing}
Let \(\varphi\in S\left(I\right)\) have a single cycle, and let \(\pi\in PM\left(\pm I\right)\).  Then \(\chi\left(\varphi,\pi\right)=2\) if and only if \(\pi\) does not connect \(I\) and \(-I\) and \(\left.\pi\right|_{I}\in S_{\mathrm{disc-nc}}\left(\varphi\right)\).

Let \(\varphi\in S\left(I\right)\) have two orbits \(V_{1}\) and \(V_{2}\), and let \(\alpha\in PM\left(\pm I\right)\) connect \(\pm V_{1}\) and \(\pm V_{2}\).  Then \(\chi\left(\varphi,\pi\right)=2\) if and only if, for some choice of sign \(\varepsilon=\pm 1\), \(\alpha\) does not connect \(V_{1}\cup\varepsilon V_{2}\) to \(\left(-V_{1}\right)\cup\left(-\varepsilon V_{2}\right)\) and \(\left.\alpha\right|_{V_{1}\cup\varepsilon V_{2}}\in S_{\mathrm{ann-nc}}\left(\left.\varphi_{+}\varphi_{-}^{-1}\right|_{V_{1}\cup\varepsilon V_{2}}\right)\).
\end{theorem}

\subsection{Matrices}

For an \(N\times N\) matrix \(X\), we denote the (usual) trace by \(\mathrm{Tr}\left(X\right):=\sum_{i=1}^{N}X_{ii}\).  We denote the normalized trace \(\mathrm{tr}\left(X\right):=\frac{1}{N}\mathrm{Tr}\left(X\right)\).

When indices appear on a matrix, we will move the subscript of a matrix to a bracketed superscript: \(X^{\left(k\right)}=X_{k}\).  We will also use negative integers to denote the transposes of the matrices: \(X^{\left(-k\right)}=X_{-k}=X_{k}^{T}\).

For \(I\subseteq\pm\left[n\right]\), we define the trace of \(N\times N\) matrices \(X_{1},\ldots,X_{n}\) along the cycles of a permutation \(\pi=\left(c_{1,1},\ldots,c_{1,n_{1}}\right)\cdots\left(c_{r,1},\ldots,c_{r,n_{r}}\right)\in S\left(I\right)\) by
\[\mathrm{Tr}_{\pi}\left(X_{1},\ldots,X_{n}\right):=\prod_{k=1}^{r}\mathrm{Tr}\left(X_{c_{k,1}}\cdots X_{c_{k,n_{k}}}\right)\textrm{.}\]
(We note that the matrices with subscripts not in \(I\) do not appear in the expression.)  The normalized trace of the matrices is defined similarly.

We will often use the following expression for products of traces of matrices, which may demonstrated by direct calculation:
\begin{lemma}
\label{lemma: traces}
Let \(I=\left\{n_{1},\ldots,n_{m}\right\}\subseteq\pm\left[n\right]\), and let \(\pi\in S\left(I\right)\).  Then for \(N\times N\) matrices \(X_{1},\ldots,X_{n}\),
\[\mathrm{Tr}_{\pi}\left(X_{1},\ldots,X_{n}\right)=\sum_{1\leq i_{n_{1}},\ldots,i_{n_{m}}\leq N}\prod_{k=1}^{m}X^{\left(n_{k}\right)}_{i_{n_{k}},i_{\pi\left(n_{k}\right)}}\textrm{.}\]
\end{lemma}
We will often write the choice of indices \(i_{n_{1}},\ldots,i_{n_{m}}\) as a function \(i:I\rightarrow\left[N\right]\).

\subsection{The Weingarten Calculus}

The Weingarten calculus was developed in \cite{MR1959915, MR2217291, MR2567222}.  We outline some of the proofs here; however, the only properties we will use are those marked as theorems.  Since we will only be using the real Weingarten function, we will use the symbol \(\mathrm{Wg}\) to refer to that function.

\begin{definition}
We define a Gram matrix indexed by \(\pi_{+},\pi_{-}\in{\cal P}_{2}\left(n\right)\) by \(G\left(\pi_{+},\pi_{-}\right)=N^{\#\left(\pi_{+}\vee\pi_{-}\right)}\).  (We will usually drop the dimension \(N\) from the notation.)  The Weingarten function \(\left[\mathrm{Wg}\left(\pi_{+},\pi_{-}\right)\right]_{\pi_{+},\pi_{-}\in{\cal P}_{2}\left(n\right)}\) is the pseudoinverse of \(G\).

The Weingarten function depends only on the size of the blocks of \(\pi_{+}\vee\pi_{-}\), each of which has an even number of elements (being the union of pairs).  If \(\pi_{+}\vee\pi_{-}\) has blocks of size \(n_{1},\ldots,n_{r}\), we may also index the Weingarten function by the Young diagram \(\lambda\vdash n/2\) with rows of length \(n_{1}/2,\ldots,n_{r}/2\).
\end{definition}

From Theorem~3.13 of \cite{MR2217291}:
\begin{theorem}
\label{theorem: leading order Weingarten}
The Weingarten function \(\mathrm{Wg}\left(\lambda\right)\) with \(\lambda=\left(\lambda_{1},\ldots,\lambda_{r}\right)\vdash n/2\) has leading-order term
\[\left(-1\right)^{n/2-r}C_{\lambda_{1}-1}\cdots C_{\lambda_{r}-1}N^{-n+\#\left(\pi_{1}\vee\pi_{2}\right)}\]
where \(C_{k}\) is the \(k\)th Catalan number \(C_{k}:=\frac{1}{k+1}\binom{2k}{k}\).
\end{theorem}
We will denote the normalized Weingarten function by
\[\mathrm{wg}\left(\pi_{+},\pi_{-}\right):=N^{n-\#\left(\pi_{+}\vee\pi_{-}\right)}\mathrm{Wg}\left(\pi_{+},\pi_{-}\right)\textrm{.}\]
We note in particular that \(\mathrm{wg}\left(1\right)=1\).

The Weingarten function can be expressed:
\[\mathrm{wg}\left(\pi_{+},\pi_{-}\right)=N^{n/2-\#\left(\pi_{+}\vee\pi_{-}\right)}\sum_{k\geq 0}\sum_{\substack{\pi_{0},\ldots,\pi_{k}\\\pi_{0}=\pi_{+},\pi_{k}=\pi_{-}\\\pi_{0}\neq\pi_{1}\neq\cdots\neq\pi_{k}}}\left(-1\right)^{k}N^{-\left|\pi_{0}\pi_{1}\right|/2-\ldots-\left|\pi_{k-1}\pi_{k}\right|/2}\]
(see \cite{MR2217291}).

The following appears in Corollary~3.4 of \cite{MR2217291}:
\begin{theorem}[Collins, \'{S}niady]
\label{theorem: integral}
Let \(O\) be a Haar-distributed orthogonal matrix.  Then
\[\mathbb{E}\left(O_{i_{1}j_{1}}\cdots O_{i_{n}j_{n}}\right)=\sum_{\substack{\left(\pi_{+},\pi_{-}\right)\in{\cal P}_{2}\left(n\right)\\i_{k}=i\circ\pi_{+},j_{k}=j\circ\pi_{-}}}\mathrm{Wg}\left(\pi_{+},\pi_{-}\right)\textrm{.}\]
\end{theorem}

\begin{definition}
For \(\rho,\sigma\in{\cal P}\left(n\right)\) with \(\pi_{+}\vee\pi_{-}\preceq\rho\preceq\sigma\), we define a sort of cumulant of the Weingarten function by
\[\prod_{V\in\sigma}\mathrm{wg}\left(\left.\pi_{+}\right|_{V},\left.\pi_{-}\right|_{V}\right)=\sum_{\tau:\rho\preceq\tau\preceq\sigma}C_{\pi_{+}\vee\pi_{-},\rho,\tau}\textrm{,}\]
or equivalently
\[C_{\pi_{+}\vee\pi_{-},\rho,\sigma}=\sum_{\rho\preceq\tau\preceq\sigma}\mu\left(\tau,\sigma\right)\prod_{V\in\tau}\mathrm{wg}\left(\left.\pi_{+}\right|_{V},\left.\pi_{-}\right|_{V}\right)\textrm{.}\]
(We use the normalized Weingarten function instead of the usual one as in \cite{MR2217291}, since it appears more often in this paper.)
\end{definition}

The following appears in Theorem~3.16 of \cite{MR2217291}:
\begin{theorem}[Collins, \'{S}niady]
\label{theorem: Weingarten cumulant}
For \(\pi_{+},\pi_{-}\in{\cal P}_{2}\left(n\right)\), and \(\rho,\sigma\in{\cal P}\left(n\right)\) with \(\pi_{+}\vee\pi_{-}\preceq\rho\preceq\sigma\), the cumulant \(C_{\pi_{+}\vee\pi_{-},\rho,\sigma}\) is of order \(N^{2\left(\#\left(\sigma\right)-\#\left(\rho\right)\right)}\).
\end{theorem}

The cumulants may be expressed
\[C_{\pi_{+}\vee\pi_{-},\rho,\sigma}=N^{n/2-\#\left(\pi_{+}\vee\pi_{-}\right)}\sum_{k\geq 0}\sum_{\substack{\pi_{0},\ldots,\pi_{k}\\\pi_{0}=\pi_{+},\pi_{k}=\pi_{-}\\\pi_{0}\neq\pi_{1}\neq\cdots\neq\pi_{k}\\\rho\vee\pi_{0}\vee\cdots\vee\pi_{k}=\sigma}}\left(-1\right)^{k}N^{-\left|\pi_{0}\pi_{1}\right|/2-\ldots-\left|\pi_{k-1}\pi_{k}\right|/2}\textrm{.}\]
To see this, we define an operator \(S\) on functions on the positive integers \(\mathbb{N}^{\ast}\rightarrow\mathbb{C}\) which are eventually zero by
\[S\left(f\right)=\sum_{k=1}^{\infty}\left(-1\right)^{k}f\left(k\right)\]
and a product \(g\ast h\) on such functions by
\[\left(g\ast h\right)\left(k\right)=\sum_{U_{1},U_{2}\subseteq\left[k\right]:U_{1}\cup U_{2}=\left[k\right]}g\left(\left|U_{1}\right|\right)h\left(\left|U_{2}\right|\right)\textrm{.}\]
The convolution \(\ast\) is associative, and \(S\left(g\ast h\right)=S\left(g\right)S\left(h\right)\) (proven in \cite{MR1959915}).  Let \(\rho,\sigma\in{\cal P}\left(n\right)\), \(\sigma=\left\{V_{1},\ldots,V_{r}\right\}\), and \(l\) a nonnegative integer.  For \(j\), \(1\leq j\leq r\), and nonnegative integers \(l_{1},\ldots,l_{r}\) with \(l_{1}+\cdots+l_{r}=l\), define \(f_{j}^{\left(l_{j}\right)}:\mathbb{N^{\ast}}\rightarrow\mathbb{C}\) by letting \(f_{j}^{\left(l_{j}\right)}\left(k\right)\) be the number of ordered \(\left(k+1\right)\)-tuples \(\left(\pi_{0},\ldots,\pi_{k}\right)\) with \(\pi_{0},\ldots,\pi_{k}\in{\cal P}_{2}\left(V_{j}\right)\), \(\pi_{0}=\left.\pi_{+}\right|_{V_{j}}\), \(\pi_{k}=\left.\pi_{-}\right|_{V_{j}}\), \(\pi_{0}\neq\pi_{1}\neq\cdots\neq\pi_{k}\), and \(\left|\pi_{0}\pi_{1}\right|+\cdots+\left|\pi_{k-1}\pi_{k}\right|=2l_{j}\), and define \(g_{j}^{\left(l_{j}\right)}\) to be the number of tuples such that, in addition, \(\left.\rho\right|_{V_{j}}\vee\pi_{0}\vee\cdots\vee\pi_{k}=\left\{V_{j}\right\}\).  Then the coefficient of \(N^{n/2-\#\left(\pi_{+}\vee\pi_{-}\right)-l}\) in \(\prod_{j=1}^{r}\mathrm{wg}\left(\left.\pi_{+}\right|_{V_{j}},\left.\pi_{-}\right|_{V_{j}}\right)\) is
\begin{multline*}
\sum_{l_{1},\ldots,l_{r}:l_{1}+\cdots+l_{r}=l}S\left(f_{1}^{\left(l_{1}\right)}\right)\cdots S\left(f_{r}^{\left(l_{r}\right)}\right)\\=\sum_{l_{1},\ldots,l_{r}:l_{1}+\cdots+l_{r}=l}S\left(f_{1}^{\left(l_{1}\right)}\ast\cdots\ast f_{r}^{\left(l_{r}\right)}\right)\textrm{,}
\end{multline*}
which is the coefficient of \(N^{n/2-\#\left(\pi_{+}\vee\pi_{-}\right)-l}\) in
\[N^{n/2-\#\left(\pi_{+}\vee\pi_{-}\right)}\sum_{k\geq 0}\sum_{\substack{\pi_{0},\ldots,\pi_{k}\\\pi_{0},\cdots\pi_{k}\preceq\sigma\\\pi_{0}=\pi_{+},\pi_{k}=\pi_{-}\\\pi_{0}\neq\pi_{1}\neq\cdots\neq\pi_{k}}}\left(-1\right)^{k}N^{-\left|\pi_{0}\pi_{1}\right|/2-\ldots-\left|\pi_{k-1}\pi_{k}\right|/2}\]
(where the subset \(U_{j}\) corresponding to \(f_{j}^{\left(l_{j}\right)}\) in the convolution is the set of \(i\in\left[k\right]\) for which \(\left.\pi_{i}\right|_{V_{j}}\neq\left.\pi_{i-1}\right|_{V_{j}}\)).  For \(\tau\in{\cal P}\left(n\right)\), the coefficient of \(N^{n/2-\#\left(\pi_{+}\vee\pi_{-}\right)-l}\) in the sum of only the terms with \(\rho\vee\pi_{0}\vee\cdots\vee\pi_{k}=\tau\) is
\[\sum_{l_{1},\ldots,l_{r}:l_{1}+\cdots+l_{r}=l}S\left(g_{1}^{\left(l_{1}\right)}\ast\cdots\ast g_{r}^{\left(l_{r}\right)}\right)=\sum_{l_{1},\ldots,l_{r}:l_{1}+\cdots+l_{r}=l}S\left(g_{1}^{\left(l_{1}\right)}\right)\cdots S\left(g_{r}^{\left(l_{r}\right)}\right)\textrm{,}\]
which is the coefficient of \(N^{n/2-\#\left(\pi_{+}\vee\pi_{-}\right)-l}\) in the given expression for the cumulants.

To put a bound on the order of the cumulant, we note that \(\left(\pi_{1},\pi_{2}\right)\mapsto\left|\pi_{1}\pi_{2}\right|\) is a distance metric.  If joining \(\rho\vee\pi_{0}\vee\pi_{j_{1}}\vee\cdots\vee\pi_{j_{r}}\vee\pi_{k}\)) (\(0<j_{1}<\cdots<j_{r}<k\) with \(\pi\) reduces the number of blocks by \(m\), then it must reduce the number of blocks of any \(\pi_{j_{s}}\) by at least as much (Lemma~\ref{lemma: useful}).  Since \(\left|\pi_{1},\pi_{2}\right|=n-2\#\left(\pi_{1}\vee\pi_{2}\right)\) (Lemma~\ref{lemma: loops} below shows that \(\#\left(\pi_{1}\pi_{2}\right)=2\#\left(\pi_{1}\vee\pi_{2}\right)\)), we have \(\left|\pi\pi_{j_{s}}\right|,\left|\pi\pi_{j_{s+1}}\right|\geq 2m\), so inserting \(\pi\) between any \(\pi_{j_{s}}\) and \(\pi_{j_{s+1}}\) increases the total distance of the path by 4m.  Thus, if we require our list of \(\pi_{j}\) to reduce the number of blocks of \(\rho\) to \(\#\left(\sigma\right)\), the total distance must be at least \(4\left(\#\left(\rho\right)-\#\left(\sigma\right)\right)\), and the bound follows.

\subsection{Freeness}

\begin{definition}
A noncommutative probability space is a pair \(\left(A,\phi_{1}\right)\) consisting of a unital algebra \(A\) and a linear functional \(\phi_{1}\) such that \(\phi_{1}\left(1_{A}\right)=1\).

A second-order probability space is a triple \(\left(A,\phi_{1},\phi_{2}\right)\), where \(\left(A,\phi_{1}\right)\) is a noncommutative probability space, and \(\phi_{2}:A^{2}\rightarrow\mathbb{C}\) is a function which is tracial in each argument such that for any \(a\in A\), \(\phi_{2}\left(1_{A},a\right)=\phi_{2}\left(a,1_{A}\right)=0\).

A real second-order probability space is a second-order probability space equipped with an involution \(a\mapsto a^{t}\) such that for any \(a,b\in A\), \(\left(ab\right)^{t}=b^{t}a^{t}\).
\end{definition}
When we consider algebras generated by elements of a real second-order probability space, we also require the algebra to be closed under this involution (which is the transpose when we are considering finite random matrices).

\begin{definition}
We say that an element \(a\in A\) of a noncommutative probability space is {\em centred} if \(\phi_{1}\left(a\right)=0\).  We may centre an element \(a\) by subtracting \(\phi_{1}\left(a\right)\) (we will denote multiples of the unit by scalars): we denote the centred element by \(\mathaccent"7017{a}:=a-\phi_{1}\left(a\right)\).
\end{definition}

\begin{definition}
We say that a word \(w:\left[n\right]\rightarrow\left[C\right]\) is {\em alternating} if \(w\left(1\right)\neq w\left(2\right)\neq\cdots\neq w\left(n\right)\).  We say that it is {\em cyclically alternating} if, in addition, \(w\left(n\right)\neq w\left(1\right)\).

For \(A_{1},\ldots,A_{C}\) subalgebras of \(A\), we say that \(a_{1},\ldots,a_{n}\) are alternating if, for \(1\leq i\leq n\), \(a_{i}\in A_{k_{i}}\) with \(k_{1}\neq k_{2}\neq\cdots\neq k_{n}\).  We say that they are cyclically alternating if, in addition, \(k_{n}\neq k_{1}\).
\end{definition}

\begin{definition}
Subspaces \(A_{1},\ldots,A_{C}\) of a noncommutative probability space are free if, for any \(a_{1},\ldots,a_{n}\) centred and alternating,
\[\phi_{1}\left(a_{1},\ldots,a_{n}\right)=0\textrm{.}\]
Subsets of \(A\) are free if the algebras they generate are free.

Subspaces \(A_{1},\ldots,A_{C}\) of a real second-order probability space are real second-order free if they are free, and for any \(a_{1},\ldots,a_{p}\) and \(b_{1},\ldots,b_{q}\) centred and cyclically alternating, we have for \(p\neq q\):
\[\phi_{2}\left(a_{1}\cdots a_{p},b_{1}\cdots b_{q}\right)=0\textrm{,}\]
and, taking subscripts modulo the appropriate range (as we will do throughout), for \(p=q\),
\[\phi_{2}\left(a_{1}\cdots a_{p},b_{1}\cdots b_{q}\right)=\sum_{k=0}^{p-1}\prod_{i=1}^{p}\varphi\left(a_{i}b_{k-i}\right)+\sum_{k=0}^{p-1}\prod_{i=1}^{p}\varphi\left(a_{i}b_{k+i}^{t}\right)\textrm{.}\]
Subsets of \(A\) are real second-order free if the algebras the generate are free.
\end{definition}

\begin{definition}
If we have a set of \(N\times N\) random matrices \(\left\{X_{\lambda}\right\}_{\lambda\in\Lambda}\) for \(N\) arbitrarily large (we will generally suppress \(N\) in the notation), we say that they have a (first-order) limit distribution if there is a noncommutative probability space \(\left(A,\phi_{1}\right)\) with \(\left\{x_{\lambda}\right\}_{\lambda\in\Lambda}\subseteq A\) such that, for any noncommutative polynomial \(p\) in \(m\) variables, we have
\[\lim_{N\rightarrow\infty}\mathbb{E}\left(\mathrm{tr}\left(p\left(X_{\lambda_{1}},\ldots,X_{\lambda_{m}}\right)\right)\right)=\phi_{1}\left(p\left(x_{\lambda_{1}},\ldots,x_{\lambda_{m}}\right)\right)\]
while any higher cumulant of such normalized traces vanishes as \(N\rightarrow\infty\).

We say that random matrices \(\left\{X_{\lambda}\right\}_{\lambda\in\Lambda}\) have a second-order limit distribution if, in addition, we have a \(\phi_{2}\) such that \(\left(A,\phi_{1},\phi_{2}\right)\) is a second-order probability space, and for \(p_{1},p_{2},\ldots\) noncommutative polynomials in \(n_{1},n_{2},\ldots\) variables, we have
\begin{multline*}
\lim_{N\rightarrow\infty}k_{2}\left(\mathrm{Tr}\left(p_{1}\left(X_{\lambda_{1,1}},\ldots,X_{1,\lambda_{n_{1}}}\right)\right),\mathrm{Tr}\left(p_{2}\left(X_{\lambda_{2,1}},\ldots,X_{2,\lambda_{n_{2}}}\right)\right)\right)\\=\phi_{2}\left(x_{\lambda_{1,1}}\cdots x_{\lambda_{1,n_{1}}},x_{\lambda_{2,1}}\cdots x_{\lambda_{2,n_{2}}}\right)\textrm{,}
\end{multline*}
while for \(r>2\), we have
\[k_{r}\left(\mathrm{Tr}\left(p_{1}\left(X_{\lambda_{1,1}},\ldots,X_{\lambda_{1,n_{1}}}\right)\right),\ldots,\mathrm{Tr}\left(p_{r}\left(X_{\lambda_{r,1}},\ldots,X_{\lambda_{r,n_{r}}}\right)\right)\right)=0\textrm{.}\]

We say that the random matrices have a real second-order limit distribution if \(\left(A,\phi_{1},\phi_{2}\right)\) is a real second-order probability space with involution \(x\mapsto x^{t}\), and for every \(\lambda\in\Lambda\), \(-\lambda\in\Lambda\) (where, as elsewhere, \(X_{-\lambda}=X_{\lambda}^{T}\) and \(x_{-\lambda}=x_{\lambda}^{t}\)).
\end{definition}

\begin{definition}
We say that sets of random matrices are {\em asymptotically free} if they have a limit distribution, and the corresponding sets of elements of the limit distribution are free.

We say that sets of random matrices are {\em asymptotically real second-order free} if they have a real second-order limit distribution, and the corresponding sets of elements of the second-order limit distribution are real second-order free.
\end{definition}

For the standard definitions in (complex) second-order freeness, see \cite{MR2216446, MR2294222, MR2302524}.

\section{Moments of Haar-Distributed Orthogonal Matrices}
\label{section: genus expansion}

\begin{definition}
For a finite set \(I\) of signed integers we define the alternating permutations \(S_{\mathrm{alt}}\left(I\right)\) as the set of permutations \(\pi\) on \(I\) such that \(\mathrm{sgn}\left(\pi\left(k\right)\right)=-\mathrm{sgn}\left(k\right)\).  We define the alternating premaps \(PM_{\mathrm{alt}}\left(\pm I\right)=PM\left(\pm I\right)\cap S_{\mathrm{alt}}\left(\pm I\right)\).
\end{definition}

\begin{lemma}
\label{lemma: loops}
For \(I\) a subset of the positive integers, the map given by \(\left(\pi_{+},\pi_{-}\right)\mapsto \pi_{-}\delta\pi_{+}\) is a bijection from \({\cal P}_{2}\left(I\right)^{2}\) to \(PM_{\mathrm{alt}}\left(\pm I\right)\).

The blocks of \(\pi_{+}\vee\pi_{-}\) each have an even number of elements, so we may construct a Young diagram whose row-lengths are half the number of elements in the blocks.  The cycles of \(\pi_{+}\pi_{-}\) consist of pairs with the same number of elements, so we may construct a Young diagram with a row for each pair of cycles whose length is the number of elements in each of those cycles.  The cycles of \(\pi_{-}\delta\pi_{+}\) each contain an even number of elements, so we may construct a Young diagram whose row-lengths are half the number of elements in the cycles of \(\pi_{-}\delta\pi_{+}/2\).  All three Young diagrams are equal.
\end{lemma}
\begin{proof}
Since neither \(\pi_{+}\) or \(\pi_{-}\) changes the sign, \(\pi_{-}\delta\pi_{+}\) is alternating.  We have \(\pi_{-}\delta\pi_{+}\left(k\right)=-\mathrm{sgn}\left(k\right)\pi_{\mathrm{sgn}\left(k\right)}\left(\left|k\right|\right)\) while its inverse has \(\pi_{+}\delta\pi_{-}\left(-k\right)=-\mathrm{sgn}\left(-k\right)\pi_{-\mathrm{sgn}\left(-k\right)}\left(\left|k\right|\right)\), and since neither \(\pi_{+}\) nor \(\pi_{-}\) have fixed points it will never take \(k\) to \(-k\), so it is a premap.  For \(k>0\), \(\pi_{-}\delta\pi_{+}\left(k\right)=-\pi_{+}\left(k\right)\) and \(\pi_{-}\delta\pi_{+}\left(-k\right)=\pi_{-}\left(k\right)\), so the map is injective.  For any alternating premap \(\rho\in PM_{\mathrm{alt}}\left(\pm I\right)\) and \(k\in I\) we let \(\pi_{+}\left(k\right)=-\rho\left(k\right)\) and \(\pi_{-}\left(k\right)=\rho\left(-k\right)\) (and we let these act trivially on \(-I\)).  Then \(\pi_{\pm}^{2}\left(k\right)=k\), and since \(\rho\left(k\right)\neq -k\), these are indeed pairings.  We calculate that \(\pi_{-}\delta\pi_{+}\left(k\right)=\rho\left(k\right)\), so the map is surjective.

Since a block of \(\pi_{+}\vee\pi_{-}\) is a union of pairs, it must have an even number of elements.

Since \(\pi_{+}\) and \(\pi_{-}\) are both self-inverse, elements of an orbit of the subgroup generated by \(\pi_{+}\) and \(\pi_{-}\) may be obtained by applying them alternatingly to an element \(k\) of that orbit.  Since \(\pi_{-}\pi_{+}\) is the inverse of \(\pi_{+}\pi_{-}\), all elements of the orbit are of the form \(\left(\pi_{+}\pi_{-}\right)^{m}\left(k\right)\) or \(\pi_{-}\left(\pi_{+}\pi_{-}\right)^{m}\left(k\right)=\left(\pi_{+}\pi_{-}\right)^{-m}\pi_{-}\left(k\right)\) for some integer \(m\).  Thus this orbit is the union of the cycle of \(\pi_{+}\pi_{-}\) containing \(k\) and its image under left multiplication by \(\pi_{-}\), which is the cycle containing \(\pi_{-}\left(k\right)\).  These cycles must be disjoint: if \(\pi_{-}\left(k\right)=\left(\pi_{+}\pi_{-}\right)^{m}\left(k\right)=\pi_{\left(-1\right)^{2m}}\cdots\pi_{\left(-1\right)^{1}}\left(k\right)\) for any \(m\), then \(\pi_{\left(-1\right)^{m}}\cdots\pi_{\left(-1\right)^{1}}\left(k\right)=\pi_{\left(-1\right)^{m+1}}\pi_{\left(-1\right)^{m}}\cdots\pi_{\left(-1\right)^{1}}\left(k\right)\), but since \(\pi_{\left(-1\right)^{m+1}}\) must not have fixed points, we have a contradiction.  Each block of \(\pi_{+}\vee\pi_{-}\) is then the disjoint union of two cycles of \(\pi_{+}\pi_{-}\) containing the same number of elements, and thus the constructed Young diagram is the same as that constructed from \(\pi_{+}\vee\pi_{-}\).

Since \(\pi_{-}\delta\pi_{+}\) is alternating, each cycle must contain an even number of cycles.  We calculate that for \(k>0\), \(\left(\pi_{-}\delta\pi_{+}\right)^{2}=\pi_{-}\pi_{+}=\left(\pi_{+}\pi_{-}\right)^{-1}\), so the positive elements of each cycle are the elements of a cycle of \(\pi_{+}\pi_{-}\), and thus the constructed Young diagram is the same as that constructed from \(\pi_{+}\pi_{-}\).
\end{proof}

We will denote the Young diagram constructed in this manner from an alternating premap \(\rho\in PM_{\mathrm{alt}}\left(\pm I\right)\) by \(\lambda\left(\rho\right)\).  For every pair of cycles in \(\rho\) of length \(n\), \(\lambda\left(\rho\right)\) has a row of length \(n/2\).

It is actually possibly to construct more correspondences between these constructions, which we will not need.  The connections between the various constructions are illustrated in the following example:

\begin{example}
\label{example: loops}

\begin{figure}
\centering
\scalebox{0.5}{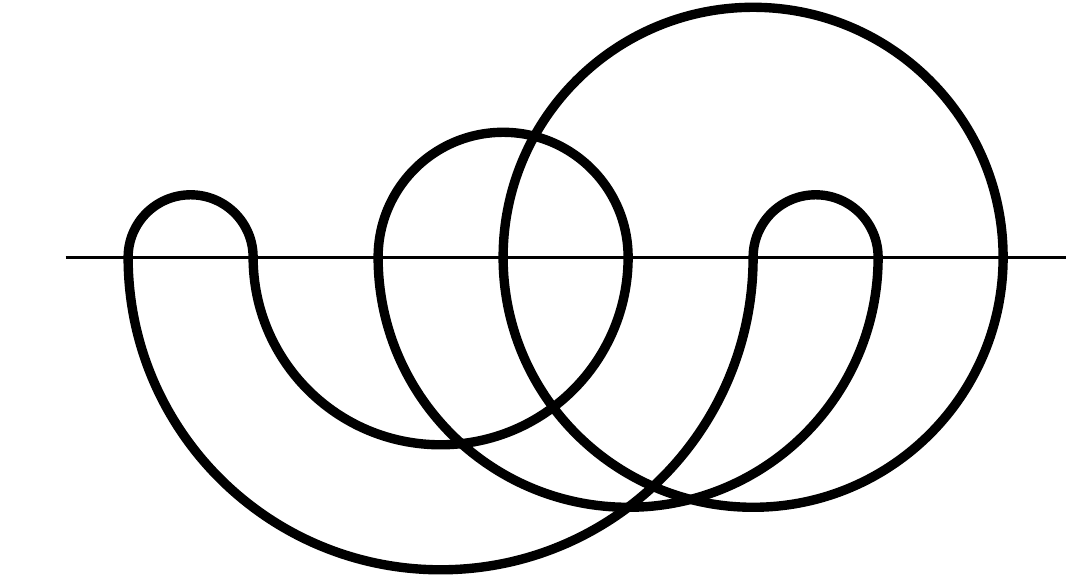}
\caption{Loops formed by two pairings.}
\label{figure: loops}
\end{figure}

In Figure~\ref{figure: loops} we show pairing \(\pi_{+}=\left(1,2\right)\left(3,5\right)\left(4,8\right)\left(6,7\right)\) above a horizontal line and \(\pi_{-}=\left(1,6\right)\left(2,5\right)\left(3,7\right)\left(4,8\right)\) below it.  Their union consists of two disjoint loops.

The blocks of \(\pi_{+}\vee\pi_{-}=\left\{\left\{1,2,3,5,6\right\},\left\{4,8\right\}\right\}\) correspond to the loops in Figure~\ref{figure: loops}.  Since one block contains six points, the other two, this partition corresponds to Young diagram \(\left(3,1\right)\).

We calculate that \(\pi_{+}\pi_{-}=\left(1,7,5\right)\left(2,3,6\right)\left(4\right)\left(8\right)\).  The first two cycles each have \(3\) elements.  Their union is the set of points in the loop with \(6\) points, and each enumerates alternating points around this loop (in opposite directions).  The last two cycles each have \(1\) element, and their union is the set of points in the loop with \(2\) points.  This permutation also corresponds to Young diagram \(\left(3,1\right)\).

We also calculate that \(\pi_{-}\delta\pi_{+}=\left(1,-2,5,-3,7,-6\right)\left(6,-7,3,-5,2,-1\right)\allowbreak\left(4,-8\right)\left(8,-4\right)\).  The two cycles of length \(6\) trace out the loop with \(6\) points in Figure~\ref{figure: loops} and the two of length \(2\) trace out the loop with \(2\) points (in opposite directions).  This premap corresponds to Young diagram \(\left(3,1\right)\).
\end{example}

\begin{proposition}
\label{proposition: genus expansion}
Let \(O\) be a Haar-distributed orthogonal matrix.  Let \(\varphi\in S_{n}\), let \(\varepsilon:\left[n\right]\rightarrow\left\{1,-1\right\}\), let \(\delta_{\varepsilon}:k\mapsto\varepsilon\left(\left|k\right|\right)k\), and let \(X_{1},\ldots,X_{n}\in M_{N\times N}\left(\mathbb{C}\right)\) be random matrices independent from \(O\).  Then
\begin{align*}
&\mathbb{E}\left[\mathrm{tr}_{\varphi}\left(O^{\varepsilon\left(1\right)}X_{1},\ldots,O^{\varepsilon\left(n\right)}X_{n}\right)\right]\\
&=\sum_{\left(\pi_{+},\pi_{-}\right)\in{\cal P}_{2}\left(n\right)^{2}}N^{\chi\left(\varphi,\delta_{\varepsilon}\pi_{-}\delta\pi_{+}\delta_{\varepsilon}\right)-2\#\left(\varphi\right)}\mathrm{wg}\left(\pi_{+},\pi_{-}\right)\\&\times\mathbb{E}\left[\mathrm{tr}_{K\left(\varphi,\delta_{\varepsilon}\pi_{-}\delta\pi_{+}\delta_{\varepsilon}\right)^{-1}/2}\left(X_{1},\ldots,X_{n}\right)\right]\\
&=\sum_{\alpha\in PM_{\mathrm{alt}}\left(\pm\left[n\right]\right)}N^{\chi\left(\varphi,\delta_{\varepsilon}\alpha\delta_{\varepsilon}\right)-2\#\left(\varphi\right)}\mathrm{wg}\left(\lambda\left(\alpha\right)\right)\\&\times\mathbb{E}\left[\mathrm{tr}_{K\left(\varphi,\delta_{\varepsilon}\alpha\delta_{\varepsilon}\right)^{-1}/2}\left(X_{1},\ldots,X_{n}\right)\right]\textrm{.}
\end{align*}
\end{proposition}
\begin{proof}
In order to accommodate arbitrary values of \(\varepsilon\), instead of the usual \(i_{k}\) and \(j_{k}\) we will use indices \(\iota_{k}\) for all \(k\in\pm\left[n\right]\), which we interpret as a function \(\iota:\pm\left[n\right]\rightarrow\left[N\right]\).  We choose indices such that the \(k\)th occurrence of the matrix \(O\) appears with the indices \(\iota_{k}\) and \(\iota_{-k}\).  By Lemma~\ref{lemma: traces}, we have
\begin{align*}
&\mathbb{E}\left[\mathrm{tr}_{\varphi}\left(O^{\varepsilon\left(1\right)}X_{1},\ldots,O^{\varepsilon\left(n\right)}X_{n}\right)\right]\\
&=\sum_{\iota:\pm\left[n\right]\rightarrow\left[N\right]}N^{-\#\left(\varphi\right)}\mathbb{E}\left(O_{\iota_{1}\iota_{-1}}X^{\left(1\right)}_{\iota_{-\varepsilon\left(1\right)1}\iota_{\varepsilon\left(\varphi\left(1\right)\right)\varphi\left(1\right)}}\cdots O_{\iota_{n}\iota_{-n}}X^{\left(n\right)}_{\iota_{-\varepsilon\left(n\right)n}\iota_{\varepsilon\left(\varphi\left(n\right)\right)\varphi\left(n\right)}}\right)\\
&=\sum_{\iota:\pm\left[n\right]\rightarrow\left[N\right]}N^{-\#\left(\varphi\right)}\mathbb{E}\left(O_{\iota_{1}\iota_{-1}}\cdots O_{\iota_{n}\iota_{-n}}\right)\mathbb{E}\left(X^{\left(1\right)}_{\iota_{\delta\delta_{\varepsilon}\left(1\right)}\iota_{\delta_{\varepsilon}\varphi\left(1\right)}}\cdots X^{\left(n\right)}_{\iota_{\delta\delta_{\varepsilon}\left(n\right)}\iota_{\delta_{\varepsilon}\varphi\left(n\right)}}\right)\textrm{.}
\end{align*}
The expected value containing entries of the matrix \(O\) may be calculated using Theorem~\ref{theorem: integral}.  On a nonvanishing term, there are constraints on the indices \(\iota_{k}=\iota_{\pi_{+}\left(k\right)}\) and \(\iota_{-k}=\iota_{\delta\pi_{-}\delta\left(-k\right)}\) for \(k>0\).  Since \(\pi_{+}\) and \(\delta\pi_{-}\delta\) act on disjoint sets, we may combine this as \(\iota_{k}=\iota_{\delta\pi_{-}\delta\pi_{+}\left(k\right)}\) for all \(k\in\pm\left[n\right]\).  Thus our expression is:
\begin{multline*}
\sum_{\iota:\pm\left[n\right]\rightarrow\left[N\right]}\sum_{\substack{\left(\pi_{+},\pi_{-}\right)\in{\cal P}_{2}\left(n\right)^{2}\\\iota=\iota\circ\delta\pi_{-}\delta\pi_{+}}}N^{-\#\left(\varphi\right)+\#\left(\pi_{-}\delta\pi_{+}\right)/2-n}\mathrm{wg}\left(\pi_{+},\pi_{-}\right)\\\times\mathbb{E}\left(X^{\left(1\right)}_{\iota_{\delta\delta_{\varepsilon}\left(1\right)}\iota_{\delta_{\varepsilon}\varphi\left(1\right)}}\cdots X^{\left(n\right)}_{\iota_{\delta\delta_{\varepsilon}\left(n\right)}\iota_{\delta_{\varepsilon}\varphi\left(n\right)}}\right)
\end{multline*}

Reversing the order of summation, we get:
\begin{multline*}
\sum_{\left(\pi_{+},\pi_{-}\right)\in{\cal P}_{2}\left(n\right)^{2}}N^{-\#\left(\varphi\right)+\#\left(\pi_{-}\delta\pi_{+}\right)/2-n}\mathrm{wg}\left(\pi_{+},\pi_{-}\right)\\\times\sum_{\substack{\iota:\pm\left[n\right]\rightarrow\left[N\right]\\\iota=\iota\circ\delta\pi_{-}\delta\pi_{+}}}\mathbb{E}\left(X^{\left(1\right)}_{\iota_{\delta\delta_{\varepsilon}\left(1\right)}\iota_{\delta_{\varepsilon}\varphi\left(1\right)}}\cdots X^{\left(n\right)}_{\iota_{\delta\delta_{\varepsilon}\left(n\right)}\iota_{\delta_{\varepsilon}\varphi\left(n\right)}}\right)\textrm{.}
\end{multline*}
We show that the constraints on the indices are those which would appear in the trace along \(K\left(\varphi,\delta_{\varepsilon}\pi_{-}\delta\pi_{+}\delta_{\varepsilon}\right)^{-1}/2=\varphi_{-}^{-1}\delta_{\varepsilon}\pi_{-}\delta\pi_{+}\delta_{\varepsilon}\varphi_{+}/2\).

We note that for \(k>0\),
\[X^{\left(k\right)}_{\iota_{\delta\delta_{\varepsilon}\varphi_{-}\left(k\right)}\iota_{\delta_{\varepsilon}\varphi_{+}\left(k\right)}}=X^{\left(k\right)}_{\iota_{\delta\delta_{\varepsilon}\left(k\right)}\iota_{\delta_{\varepsilon}\varphi\left(k\right)}}=X^{\left(-k\right)}_{\iota_{\delta\delta_{\varepsilon}\varphi_{-}\left(-k\right)}\iota_{\delta_{\varepsilon}\varphi_{+}\left(-k\right)}}\textrm{,}\]
so the product of terms from matrices \(X_{k}\) may be expressed
\[\prod_{k\in\varphi_{-}^{-1}\delta_{\varepsilon}\pi_{-}\delta\pi_{+}\delta_{\varepsilon}\varphi_{+}/2}X^{\left(k\right)}_{\iota_{\delta\delta_{\varepsilon}\varphi_{-}\left(k\right)}\iota_{\delta_{\varepsilon}\varphi_{+}\left(k\right)}}\textrm{.}\]
The constraint \(\iota=\iota\circ\delta\pi_{-}\delta\pi_{+}\) tells us that the second index of \(X_{k}\) is equal to \(\iota_{\delta\pi_{-}\delta\pi_{+}\delta_{\varepsilon}\varphi_{+}\left(k\right)}\), the first index of \(X_{\varphi_{-}^{-1}\delta_{\varepsilon}\pi_{+}\delta\pi_{-}\delta_{\varepsilon}\varphi_{+}}\).  (Taking the constraint in the opposite direction, it is also equal to \(\iota_{\pi_{+}\delta\pi_{-}\delta\delta_{\varepsilon}\varphi_{+}\left(k\right)}\), the first index on \(X_{\varphi_{-}^{-1}\delta_{\varepsilon}\delta\pi_{+}\delta\pi_{-}\delta\delta_{\varepsilon}\varphi_{+}}\), and since \(\pi_{-}\delta\pi_{+}\) is a premap, this is the same term.  Likewise, applying the constraints to the first indices results in the same equality.  Since any subscript of \(\iota\) appears at most once, these are all of the constraints.)  Relabelling the first index of \(X_{k}\) as \(i_{k}\), and replacing the second index with the first index which it is constrained to be equal to, we find that these are exactly the constraints on the indices of
\[\mathrm{Tr}_{K\left(\varphi,\delta_{\varepsilon}\pi_{-}\delta\pi_{+}\delta_{\varepsilon}\right)^{-1}/2}\left(X_{1},\ldots,X_{n}\right)\textrm{.}\]
Expressing this in terms of the normalized trace and Weingarten function, the first expression follows, and the second follows from Lemma~\ref{lemma: loops}.
\end{proof}

By Lemma~\ref{lemma: Euler characteristic} and Theorems~\ref{theorem: unoriented noncrossing} and \ref{theorem: leading order Weingarten}, if \(X_{1},\ldots,X_{n}\) have a limit distribution, the large matrix limit of the expected value in Proposition~\ref{proposition: genus expansion} exists, and the surviving terms are those where \(\alpha\) restricts to an orientable disconnected (disc) noncrossing permutation on each cycle of \(\varphi\).  At highest order, the contribution of each hyperedge and each vertex contributes multiplicatively.

\begin{example}
\label{example: moment}

If we wish to calculate
\[\mathbb{E}\left(\mathrm{tr}\left(OX_{1}OX_{2}O^{T}X_{3}\right)\mathrm{tr}\left(OX_{4}O^{T}X_{5}O^{T}X_{6}OX_{7}OX_{8}\right)\right)\textrm{,}\]
we construct the faces shown in Figure~\ref{figure: faces}.  Indices appear around the faces in the cyclic order they appear in the traces.  If we wish to consider the term corresponding to the pairings in Example~\ref{example: loops}, we glue the edges corresponding to indices constrained to be equal (also shown in Figure~\ref{figure: faces}).  The full expression is a sum over all possible surfaces constructed by gluing pairs of positive index edges and pairs of negative index edges.  The direction of the gluing is such that the corners with the \(X_{k}\) matrices are together.

We find a number of vertices whose corners contain the matrices \(X_{k}\).  The constraints on their indices are those that Lemma~\ref{lemma: traces} gives for the trace of the product of the matrices (in clockwise order).  Expressed in terms of the normalized trace, each of these vertices contributes an order of \(N\).  The subscripts appearing around these vertices (in counterclockwise order) are given by the cycles of the permutation \(K\left(\varphi,\delta_{\varepsilon}\pi_{-}\delta\pi_{+}\delta_{\varepsilon}\right)^{-1}\), so the trace is along the cycles of \(K\left(\varphi,\delta_{\varepsilon}\pi_{-}\delta\pi_{+}\delta_{\varepsilon}\right)^{-1}/2\).  In this case, \(K\left(\varphi,\delta_{\varepsilon}\pi_{-}\delta\pi_{+}\delta_{\varepsilon}\right)^{-1}=\left(1,-3,5\right)\left(-5,3,-1\right)\left(2,7,-8,4\right)\left(-4,9,-7,-2\right)\left(6\right)\left(-6\right)\).

In addition, we find a number of vertices which contain the matrix \(O\) or its transpose.  These correspond to the loops of \(\pi_{+}\vee\pi_{-}\) (see Example~\ref{example: loops}).  By Theorem~\ref{theorem: leading order Weingarten}, each of these vertices also contributes an order of \(N\), and together a factor of \(\mathrm{wg}\left(\lambda\left(\pi_{-}\delta\pi_{+}\right)\right)\).  At highest order, each vertex contributes a factor of \(\left(-1\right)^{k/2-1}C_{k/2-1}\), where \(k\) is the degree of the vertex and \(C_{k/2-1}\) is the \(\left(k/2-1\right)\)th Catalan number.  We have \(\lambda\left(\pi_{-}\delta\pi_{+}\right)=\left[3,1\right]\), and (from \cite{MR2217291}),
\begin{multline*}
\mathrm{wg}\left(\left[3,1\right]\right)=N^{8-2}\mathrm{Wg}\left(\left[3,1\right]\right)\\=\frac{2N^{6}}{\left(N+1\right)\left(N+2\right)\left(N+6\right)\left(N-1\right)\left(N-2\right)\left(N-3\right)}
\end{multline*}
which, when \(N\rightarrow\infty\), is \(\left(-1\right)^{4-2}C_{3-1}C_{1-1}=2\).

\begin{figure}
\centering
\scalebox{0.75}{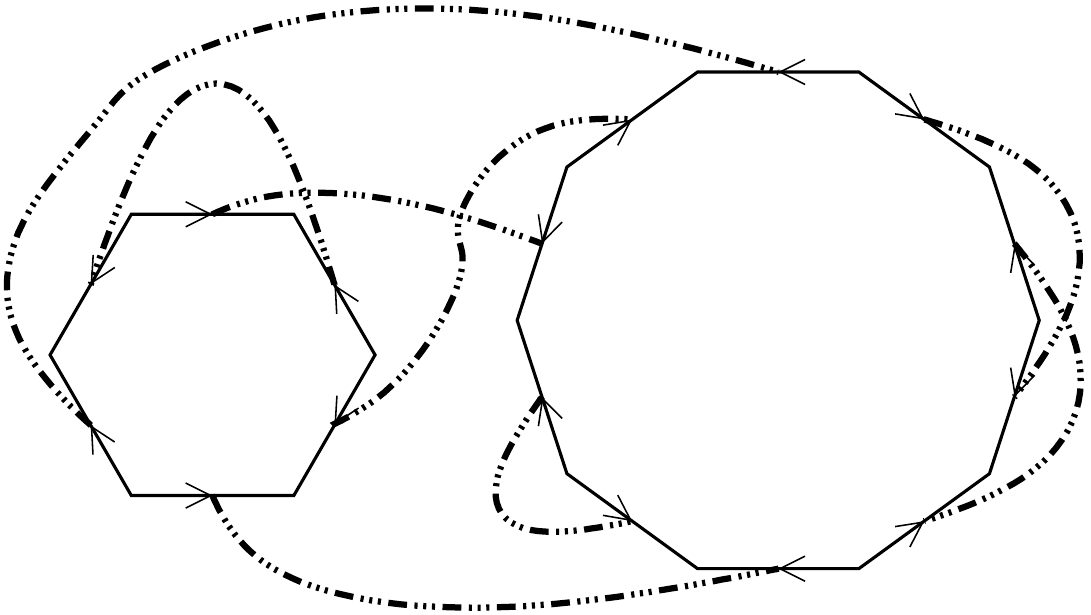}
\caption{Faces constructed for the calculation in Example~\ref{example: moment}.}
\label{figure: faces}
\end{figure}

\begin{figure}
\centering
\scalebox{0.75}{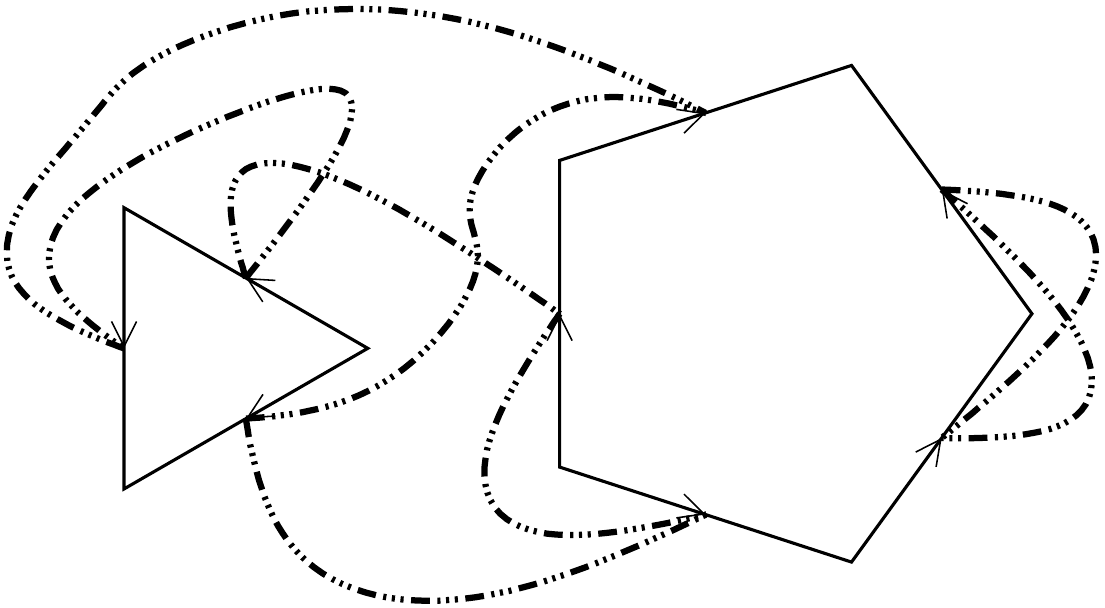}
\caption{Example~\ref{example: moment} drawn with hyperedges.}
\label{figure: hyperedges}
\end{figure}

We may also interpret the hypermap \(\pi_{-}\delta\pi_{+}=\alpha\) as the set of hyperedges (edges which may have any positive integer number of ends, rather than just two, which may be visualized as vertices which alternate with the original set of vertices as in the original construction in Figure~\ref{figure: faces}; see \cite{MR0404045, MR2036721, 2012arXiv1204.6211R}).  This allows us to represent each occurrence of the matrix \(O\) or \(O^{T}\) as a single edge, and each \(X_{k}\) as a vertex.  In Figure~\ref{figure: hyperedges} we represent them as shapes touching the edges they are gluing.  Twists in the shape represent the edge being identified in a reversed direction.  This identification must be consistent with the directions marked with arrows in the diagram, which is counter-clockwise for \(O\) terms and clockwise for \(O^{T}\) terms.  The cycles of \(\alpha\) (divided by two) give the partition the Weingarten function is applied to, and the vertices again give the trace of the \(X_{k}\).

We calculate that \(\chi\left(\varphi,\delta_{\varepsilon}\alpha\delta_{\varepsilon}\right)=2+2+3-8=-1\), so the surface is a connected sum of three projective planes (cross-cap genus \(3\)).

We calculate that the contribution of this diagram to the expected value is
\[\frac{2N\mathbb{E}\left[\mathrm{tr}\left(X_{1}X_{3}^{T}X_{5}\right)\mathrm{tr}\left(X_{2}X_{7}X_{8}^{T}X_{4}\right)\mathrm{tr}\left(X_{6}\right)\right]}{\left(N+1\right)\left(N+2\right)\left(N+6\right)\left(N-1\right)\left(N-2\right)\left(N-3\right)}\textrm{.}\]
\end{example}

Proposition~\ref{proposition: genus expansion} shows that Haar-distributed orthogonal matrices satisfy a number of the hypotheses of the matrices discussed in \cite{2011arXiv1101.0422R}; specifically, for all \(\varphi\in S_{n}\), \(\varepsilon:\left[n\right]\rightarrow\left\{1,-1\right\}\), and \(X_{k}\) random matrices independent from \(O\), we have
\begin{multline*}
\mathbb{E}\left(\mathrm{tr}_{\varphi}\left(O^{\left(\varepsilon\left(1\right)\right)}X_{1},\ldots,O^{\left(\varepsilon\left(n\right)\right)}X_{n}\right)\right)\\=\sum_{\alpha\in PM_{\mathrm{alt}}\left(\pm\left[n\right]\right)}N^{\chi\left(\varphi,\delta_{\varepsilon}\alpha\delta_{\varepsilon}\right)-2\#\left(\varphi\right)}\left(\mathrm{wg}\circ\lambda\right)\left(\alpha\right)\\\times\mathbb{E}\left[\mathrm{tr}_{K\left(\varphi,\delta_{\varepsilon}\alpha\delta_{\varepsilon}\right)^{-1}/2}\left(X_{1},\ldots,X_{n}\right)\right]\textrm{;}
\end{multline*}
for each finite set of positive integers \(I\), \(PM_{\mathrm{alt}}\left(\pm I\right)\subseteq PM\left(\pm I\right)\) is a subset of the premaps on \(\pm I\) such that for any \(J\subseteq I\), the \(\alpha\in PM_{\mathrm{alt}}\left(\pm I\right)\) which do not connect \(\pm J\) and \(\pm\left(I\setminus J\right)\) are the product of a \(\alpha_{1}\in PM_{\mathrm{alt}}\left(\pm J\right)\) and \(\alpha_{2}\in PM_{\mathrm{alt}}\left(\pm\left(I\setminus J\right)\right)\); and \(\mathrm{wg}\circ\lambda:\bigcup_{I\subseteq\mathbb{N},\left|I\right|<\infty}PM_{\mathrm{alt}}\left(\pm I\right)\rightarrow\mathbb{C}\) is a function such that \(\lim_{N\rightarrow\infty}\mathrm{wg}\circ\lambda\left(\alpha\right)\) exists.  (The only hypothesis that is not satisfied is that \(\mathrm{wg}\circ\lambda\) be multiplicative in disconnected factors.)  In particular, this means (see \cite{2011arXiv1101.0422R}) that, if \(O_{1},\ldots,O_{C}\) are independent Haar-distributed orthogonal matrices and \(w:\left[n\right]\rightarrow\left[C\right]\) is a word, then
\begin{multline*}
\mathbb{E}\left[\mathrm{tr}_{\varphi}\left(O_{w\left(1\right)}^{\left(\varepsilon\left(1\right)\right)}X_{1},\ldots,O_{w\left(n\right)}^{\left(\varepsilon\left(n\right)\right)}X_{n}\right)\right]\\=\sum_{\substack{\alpha=\alpha_{1},\ldots,\alpha_{C}\\\alpha_{c}\in PM_{\mathrm{alt}}\left(\pm w^{-1}\left(c\right)\right)}}N^{\chi\left(\varphi,\delta_{\varepsilon}\alpha\delta_{\varepsilon}\right)-2\#\left(\varphi\right)}\mathrm{wg}\left(\lambda\left(\alpha_{1}\right)\right)\cdots\mathrm{wg}\left(\lambda\left(\alpha_{C}\right)\right)\\\times\mathbb{E}\left(\mathrm{tr}_{K\left(\varphi,\delta_{\varepsilon}\alpha\delta_{\varepsilon}\right)^{-1}/2}\left(X_{1},\ldots,X_{n}\right)\right)\textrm{.}
\end{multline*}
We may depict this diagrammatically by considering only gluings where each hyperedge may connect only copies of the same matrix, and the contribution of the hyperedges is calculated for each matrix \(O_{c}\) separately.  Since \(\mathrm{wg}\circ\lambda\) is not multiplicative, we can no longer express the contribution of a diagram as the product of the contributions of its connected components.

\section{Cumulants with Haar-Distributed Orthogonal Matrices}
\label{section: cumulants}

\begin{proposition}
\label{proposition: cumulants}
Let \(n_{1},\ldots,n_{r}\geq 0\) with \(n:=n_{1},\ldots,n_{r}\).  Let \(O_{1},\ldots,O_{C}\) be independent Haar-distributed orthogonal matrices, and let \(X_{1},\ldots,X_{n}\) be random matrices independent from the \(O_{c}\).  Let \(w:\left[n\right]\rightarrow\left[C\right]\), let \(\varepsilon:\left[n\right]\rightarrow\left\{1,-1\right\}\), and for \(1\leq k\leq r\), let
\[Y_{k}=\mathrm{Tr}\left(O_{w\left(n_{1}+\cdots+n_{k-1}+1\right)}^{\varepsilon\left(n_{1}+\cdots+n_{k-1}+1\right)}X_{n_{1}+\cdots+n_{k-1}+1}\cdots O_{w\left(n_{1}+\cdots+n_{k}\right)}^{\varepsilon\left(n_{1}+\cdots+n_{k}\right)}X_{n_{1}+\cdots+n_{k}}\right)\textrm{.}\]
Let \(\varphi=\left(1,\ldots,n_{1}\right)\cdots\left(n_{1}+\cdots+n_{r-1}+1,\ldots,n\right)\).  For \(\alpha=\pi_{-}\delta\pi_{+}\in PM_{\mathrm{alt}}\left(\pm\left[n\right]\right)\) let \(\pi=\pi_{+}\vee\pi_{-}\).  Let \(\sigma_{1},\ldots,\sigma_{s}\) be the particular cycles of \(K\left(\varphi,\delta_{\varepsilon}\alpha\delta_{\varepsilon}\right)\).  (We omit the indexing of \(s\) and the \(\sigma_{k}\) by \(\alpha\) to avoid clutter.)  For \(\tau\in{\cal P}\left(s\right)\), let \(\tau_{\sigma}\in{\cal P}\left(n\right)\) have a block corresponding to each block \(V\in\rho\) equal to \(\left\{\left|i\right|:i\in\sigma_{k},k\in V\right\}\).  Then:
\begin{multline*}
k_{r}\left(Y_{1},\ldots,Y_{r}\right)\\=\sum_{\substack{\alpha=\alpha_{1}\cdots\alpha_{C}\\\alpha_{c}\in PM_{\mathrm{alt}}\left(\pm w^{-1}\left(c\right)\right)}}N^{\chi\left(\varphi,\delta_{\varepsilon}\alpha\delta_{\varepsilon}\right)-r}\sum_{\substack{\rho:\pi\preceq\rho\preceq\ker\left(w\right)\\\tau\in{\cal P}\left(s\right)\\\varphi\vee\rho\vee\tau_{\sigma}=1_{n}}}C_{\pi,\pi,\rho}\\\times k_{\tau}\left(\mathrm{tr}_{\sigma_{1}^{-1}}\left(X_{1},\ldots,X_{n}\right),\ldots,\mathrm{tr}_{\sigma_{s}^{-1}}\left(X_{1},\ldots,X_{n}\right)\right)\textrm{.}
\end{multline*}

If the algebra generated by the \(X_{k}\) has a second-order limit distribution, then \(\lim_{N\rightarrow\infty}k_{2}\left(Y_{1},Y_{2}\right)\) exists, and for \(r>2\), \(\lim_{N\rightarrow\infty}k_{r}\left(Y_{1},\ldots,Y_{r}\right)=0\).
\end{proposition}
\begin{proof}
The cumulant may be written
\begin{multline*}
\sum_{\substack{\upsilon\in{\cal P}\left(n\right)\\\upsilon\succeq\varphi}}\mu\left(\upsilon,1_{n}\right)a_{\upsilon}\left(Y_{1},\ldots,Y_{r}\right)\\=\sum_{\substack{\upsilon\in{\cal P}\left(n\right)\\\upsilon\succeq\varphi}}\mu\left(\upsilon,1_{n}\right)\sum_{\substack{\alpha=\alpha_{1}\cdots\alpha_{C}\\\alpha_{c}\in PM_{\mathrm{alt}}\left(\pm w^{-1}\left(c\right)\right)\\\pi\preceq\upsilon}}N^{\chi\left(\varphi,\delta_{\varepsilon}\alpha\delta_{\varepsilon}\right)-r}\\\times\prod_{V\in\tau}\mathrm{wg}\left(\lambda\left(\left.\alpha_{1}\right|_{\pm V}\right)\right)\cdots\mathrm{wg}\left(\lambda\left(\left.\alpha_{C}\right|_{\pm V}\right)\right)\\\times\mathbb{E}\left[\mathrm{tr}_{\left.K\left(\varphi,\delta_{\varepsilon}\alpha\delta_{\varepsilon}\right)^{-1}/2\right|_{\pm V}}\left(X_{1},\ldots,X_{n}\right)\right]\textrm{.}
\end{multline*}
We may then expand the Weingarten functions and the expected value in terms of their cumulants.  The innermost product becomes:
\begin{multline*}
\left[\sum_{\rho:\pi\preceq\rho\preceq\upsilon\wedge\ker\left(w\right)}C_{\pi,\pi,\rho}\right]\\\times\left[\sum_{\substack{\tau\in{\cal P}\left(s\right)\\\tau_{\sigma}\preceq\upsilon}}k_{\tau}\left(\mathrm{tr}_{\sigma_{1}^{-1}}\left(X_{1},\ldots,X_{n}\right),\ldots,\mathrm{tr}_{\sigma^{-1}_{s}}\left(X_{1},\ldots,X_{n}\right)\right)\right]
\\=\sum_{\substack{\rho:\pi\preceq\rho\preceq\tau\wedge\ker\left(w\right)\\\tau:\tau_{\sigma}\preceq\upsilon}}C_{\pi,\pi,\rho}k_{\tau}\left(\mathrm{tr}_{\sigma^{-1}_{1}}\left(X_{1},\ldots,X_{n}\right),\ldots,\mathrm{tr}_{\sigma^{-1}_{s}}\left(X_{1},\ldots,X_{n}\right)\right)\textrm{.}
\end{multline*}
Permuting the order of summation
\begin{multline*}
\sum_{\alpha\in PM_{\mathrm{alt}}\left(\pm\left[n\right]\right)}N^{\chi\left(\varphi,\delta_{\varepsilon}\alpha\delta_{\varepsilon}\right)-r}\sum_{\substack{\rho:\pi\preceq\rho\preceq\ker\left(w\right)\\\tau\in{\cal P}\left(s\right)}}C_{\pi,\pi,\rho}\\\times k_{\tau}\left(\mathrm{tr}_{\sigma_{1}^{-1}}\left(X_{1},\ldots,X_{n}\right),\ldots,\mathrm{tr}_{\sigma^{-1}_{s}}\left(X_{1},\ldots,X_{n}\right)\right)\sum_{\substack{\upsilon\in{\cal P}\left(n\right)\\\upsilon\succeq\varphi\vee\rho\vee\tau_{\sigma}}}\mu\left(\upsilon,1_{n}\right)\textrm{.}
\end{multline*}
The innermost sum is equal to \(1\) when the lower bound \(\varphi\vee\rho\vee\tau_{\sigma}=1_{n}\), and \(0\) otherwise.  The first part of the theorem follows.

We know that \(C_{\pi,\pi,\rho}=O\left(N^{2\left(\#\left(\rho\right)-\#\left(\pi\right)\right)}\right)\) and
\[k_{\tau}\left(\mathrm{tr}_{\sigma_{1}^{-1}}\left(X_{1},\ldots,X_{n}\right),\ldots,\mathrm{tr}_{\sigma^{-1}_{s}}\left(X_{1},\ldots,X_{n}\right)\right)=O\left(N^{-\left|\left\{k:\left\{k\right\}\notin\tau\right\}\right|}\right)\textrm{,}\]
and vanishes asymptotically if \(\tau\) contains a block of size greater than \(2\).  The full cumulant is then \(O\left(N^{\chi\left(\varphi,\delta_{\varepsilon}\alpha\delta_{\varepsilon}\right)-r+2\left(\#\left(\rho\right)-\#\left(\pi\right)\right)-\#\left\{k:\left\{k\right\}\notin\tau\right\}}\right)\).  We know that \(\chi\left(\varphi,\delta_{\varepsilon}\alpha\delta_{\varepsilon}\right)\leq 2\#\left(\varphi\vee\pi\right)\), so
\begin{multline*}
\chi\left(\varphi,\delta_{\varepsilon}\alpha\delta_{\varepsilon}\right)-r+2\left(\#\left(\rho\right)-\#\left(\pi\right)\right)-\left|\left\{k:\left\{k\right\}\notin\tau\right\}\right|\\\leq 2\#\left(\varphi\vee\pi\right)-r+2\left(\#\left(\rho\right)-\#\left(\pi\right)\right)-\left|\left\{k:\left\{k\right\}\notin\tau\right\}\right|\textrm{.}
\end{multline*}
By Lemma~\ref{lemma: useful}, \(\#\left(\pi\right)-\#\left(\rho\right)\geq\#\left(\varphi\vee\pi\right)-\#\left(\varphi\vee\pi\vee\rho\right)\).  Thus
\begin{multline}
\chi\left(\varphi,\delta_{\varepsilon}\alpha\delta_{\varepsilon}\right)-r+2\left(\#\left(\rho\right)-\#\left(\pi\right)\right)-\left|\left\{k:\left\{k\right\}\notin\tau\right\}\right|\\\leq 2\#\left(\varphi\vee\pi\vee\rho\right)-r-\left|\left\{k:\left\{k\right\}\notin\sigma\right\}\right|\textrm{.}
\label{formula: cumulant inequality}
\end{multline}
Since \(\tau_{\sigma}\) must connect the blocks of \(\varphi\vee\pi\vee\rho\), \(\tau\) must have a \(k\) not in a singleton in each block of \(\varphi\vee\pi\vee\rho\), or
\begin{multline}
\chi\left(\varphi,\delta_{\varepsilon}\alpha\delta_{\varepsilon}\right)-r+2\left(\#\left(\rho\right)-\#\left(\pi\right)\right)-\left|\left\{k:\left\{k\right\}\notin\tau\right\}\right|\\\leq\#\left(\varphi\vee\pi\vee\rho\right)-r\textrm{.}
\label{formula: Weingarten cumulant inequality}
\end{multline}
Since \(r=\#\left(\varphi\right)\), we have that the cumulant is \(O\left(1\right)\).

Furthermore, if \(r>2\), then either \(\varphi\vee\pi=1_{n}\); or \(\#\left(\varphi\vee\pi\vee\rho\right)>\#\left(\varphi\vee\pi\right)\geq r\), in which case we can get a strict inequality from (\ref{formula: Weingarten cumulant inequality}); or \(\tau_{\sigma}\) must connect more than \(2\) blocks, in which case it either has at least two \(k\) not in singletons in one block of \(\varphi\vee\pi\vee\rho\), or a block of size greater than \(2\).  In the first case \(\chi\left(\varphi,\delta_{\varepsilon}\alpha\delta_{\varepsilon}\right)-r<0\).  In the second, the inequality we get from (\ref{formula: cumulant inequality}) is strict.  In the third, if we have more than on \(k\) not in a singleton then the inequality (\ref{formula: Weingarten cumulant inequality}) is strict, and if \(\sigma\) has a block of size larger than \(2\), the cumulant over the block of size greater than \(2\), and hence the product, vanishes.
\end{proof}

\section{Matrices Orthogonally in General Position}
\label{section: freeness}

We present two technical lemmas which will be useful in the following proofs.  Lemma~\ref{lemma: neighbours, disc} shows that any disc-noncrossing permutation with no singletons must have a cycle which contains a pair of neighbours: since every cycle must contain more than one element, we can find a sequence of nested cycles until we find one containing the pair of neighbours.  (In fact we can find two, since we may always find two distinct notions of ``outside'', and hence two directions of nesting.)  Lemma~\ref{lemma: neighbours, annulus} shows that if an annular-noncrossing permutation with no singletons has a cycle with more than one element in one disc of the annulus, then it must have a cycle with a pair of neighbours: elements on one side of this permutation may be connected to the other cycle of \(\varphi\), but we may similarly find a sequence of nested permutations on the other side until we find a pair of neighbours.  This means that in alternating products of centred matrices orthogonally in general position, highest order terms must correspond to an \(\alpha\) connecting a pair of neighbours under these conditions.  Roughly, it may not connect a \(2n\) to \(\varphi\left(2n\right)\), since terms are alternating, and if it connects a \(2n-1\) to \(\varphi\left(2n-1\right)\), the centred matrix between them appears alone in a cycle of \(K\left(\varphi,\delta^{\prime}\alpha\delta^{\prime}\right)^{-1}\), allowing us to show that the term must vanish.  In the first-order case, all terms vanish, and in the second-order case, we are left with the spoke diagrams, as we expected.

\begin{lemma}
\label{lemma: neighbours, disc}
Let \(\varphi\in S\left(n\right)\) be a permutation with one cycle.  If \(\alpha\in S_{\mathrm{disc-nc}}\left(\varphi\right)\) has no cycles with only one element, then there are at least two distinct \(k\in\left[n\right]\) such that \(\alpha^{-1}\left(k\right)=\varphi\left(k\right)\).
\end{lemma}
\begin{proof}
For any \(a\), \(b:=\alpha^{-1}\left(a\right)\neq a\).  If \(\alpha^{-1}\left(a\right)\neq\varphi\left(a\right)\), then let \(c=\varphi\left(a\right)\).  Point \(c\) must be in a distinct cycle of \(\pi\) from \(a\) and \(b\): if not, then \(\left.\varphi\right|_{\left\{a,b,c\right\}}=\left.\alpha\right|_{\left\{a,b,c\right\}}=\left(a,c,b\right)\), so \(\alpha\) would be disc nonstandard.  Let \(d=\alpha^{-1}\left(c\right)\neq c\).  Since \(\left.\alpha\right|_{\left\{a,b,c,d\right\}}=\left(a,b\right)\left(c,d\right)\), we must have \(\left.\varphi\right|_{\left\{a,b,c,d\right\}}=\left(a,c,d,b\right)\) (disc-noncrossing condition).  So if \(m_{1}\) and \(m_{2}\) are the smallest positive integers such that \(\varphi^{m_{1}}\left(a\right)=b\) and \(\varphi^{m_{2}}\left(c\right)=d\), then \(m_{2}<m_{1}\).  By induction, we can repeat this process and eventually find a \(k\) such that \(\alpha^{-1}\left(k\right)=\varphi\left(k\right)\).

If we now let \(a^{\prime}=b\) and \(b^{\prime}=\alpha^{-1}\left(a^{\prime}\right)\), then we may find a \(k^{\prime}\) such that \(\alpha^{-1}\left(k^{\prime}\right)=\varphi\left(k^{\prime}\right)\).  Since \(\left.\varphi\right|_{\left\{a,b,k,k^{\prime}\right\}}=\left(a,k,b,k^{\prime}\right)\), we know that \(k\neq k^{\prime}\).
\end{proof}

\begin{lemma}
\label{lemma: neighbours, annulus}
Let \(\varphi\in S\left(n\right)\) be a permutation with two cycles.  If \(\alpha\in S_{\mathrm{ann-nc}}\left(\varphi\right)\), then if there are any two points sharing both a cycle of \(\varphi\) and one of \(\alpha\), we can find a \(k\) such that \(\alpha^{-1}\left(k\right)=\varphi\left(k\right)\).
\end{lemma}
\begin{proof}
Let \(a\) and \(b\) share cycles of both \(\varphi\) and \(\pi\).  We will denote this cycle of \(\varphi\) by \(\varphi_{\mathrm{ext}}\).  If there are points \(c\) and \(d\) in the other cycle of \(\varphi\) such that \(\left.\alpha\right|_{\left\{a,b,c,d\right\}}=\left(a,c,b,d\right)\), then \(\pi\) is annular nonstandard (second condition).  Thus we can find \(a,b\in\varphi_{\mathrm{ext}}\) such that \(b=\alpha^{-1}\left(a\right)\).

We choose \(a\) and \(b\) such that, in addition, \(b=\varphi^{m}\left(a\right)\) for the minimal positive integer \(m\).  If \(b\neq\varphi\left(a\right)\), then let \(x=\varphi\left(a\right)\).  The point \(x\) must be in a different cycle of \(\alpha\) from \(a\) and \(b\): \(\left.\varphi\right|_{\left\{a,b,x\right\}}=\left(a,x,b\right)\), so \(\left.\alpha\right|_{\left\{a,b,x\right\}}=\left(a,x,b\right)\) would be disc nonstandard.  By the first annular-noncrossing condition and the minimality of \(m\), \(y=\alpha^{-1}\left(x\right)\in\varphi_{\mathrm{int}}\).

Consider an element \(c\) with \(\varphi_{\left\{a,b,c\right\}}=\left(a,b,c\right)\).  Element \(c\) cannot share a cycle with \(x\) (first annular-noncrossing condition), so if \(I\) is the union of the orbits of such \(c\), then \(\left.\alpha\right|_{I}\in S_{\mathrm{disc-nc}}\left(\left.\lambda_{x,y}\right|_{I}\right)\) (second and third annular-noncrossing conditions).  This means that all \(d\) sharing cycles with a \(c\in I\) have \(\left.\lambda_{x,y}\right|_{\left\{a,b,d\right\}}=\left(a,b,d\right)\) (third annular-noncrossing condition), and thus also have \(\varphi_{\left\{a,b,d\right\}}=\left(a,b,d\right)\) (and \(\left.\alpha\right|_{I}\) is the restriction of \(\alpha\) to \(I\)).

We can then apply Lemma~\ref{lemma: neighbours, disc} to \(\left.\lambda_{x,y}\right|_{I}\), so we must be able to find two distinct \(k\in I\) such that \(\alpha^{-1}\left(k\right)=\left.\lambda_{x,y}\right|_{I}\left(k\right)\).  At least one of these must not be \(\varphi^{-1}\left(a\right)\), and since \(\lambda_{x,y}\) and \(\varphi\) agree on all other \(c\in I\), we have the desired \(k\).
\end{proof}

The following lemma shows that, for independent \(O_{1},\ldots,O_{C}\) indexed by some set of colours \(\left[C\right]\), in the expansion of an expression of the form
\begin{multline*}
\mathbb{E}\left[\mathrm{tr}\left(O_{w\left(1\right)}^{T}X_{1}O_{w\left(1\right)}\cdots O_{w\left(n_{1}\right)}^{T}X_{n_{1}}O_{w\left(n_{1}\right)}\right)\cdots\right.\\\left.\mathrm{tr}\left(O_{w\left(n_{1}+\cdots+n_{r-1}\right)}^{T}X_{n_{1}+\cdots+n_{r-1}+1}O_{w\left(n_{1}+\cdots+n_{r-1}\right)}\right)\right)]
\end{multline*}
according to Proposition~\ref{proposition: genus expansion}, the cycles of any \(K\left(\varphi,\delta_{\varepsilon}\alpha\delta_{\varepsilon}\right)\) indexing the \(X_{k}\) have a consistent colour.  This means that any single trace appearing in a trace along \(K\left(\varphi,\delta_{\varepsilon}\alpha\delta_{\varepsilon}\right)\) must satisfy the hypotheses of the limit distributions of that algebra.

\begin{lemma}
\label{lemma: colours}
Let \(n_{1},\ldots,n_{r}\) be even positive integers, and let \(n:=n_{1}+\cdots+n_{r}\), let \(\delta^{\prime}:k\mapsto\left(-1\right)^{k}k\), and let \(w:\left[n\right]\rightarrow\left[C\right]\) be a word in colours \(\left[C\right]\) such that for positive integer \(k\), \(w\left(2k-1\right)=w\left(2k\right)\).  Then for any \(\alpha\in PM_{\mathrm{alt}}\left(\left[n\right]\right)\), the permutation \(K\left(\varphi,\delta^{\prime}\alpha\delta^{\prime}\right)^{-1}\) takes odd integers to odd integers, and for any odd integer \(k\), \(w\left(\left|K\left(\varphi,\delta^{\prime}\pi\delta^{\prime}\right)^{-1}\left(k\right)\right|\right)=w\left(\left|k\right|\right)\).
\end{lemma}
\begin{proof}
Let \(k\in\pm\left[n\right]\) be odd, and consider the action of each permutation in \(\varphi_{-}^{-1}\delta^{\prime}\alpha\delta^{\prime}\varphi_{+}\left(k\right)\).  If \(\alpha\) does not change the parity, then neither or both \(\delta^{\prime}\) terms act.  Since \(\alpha\) is alternating, the sign is changed, so neither or both of \(\varphi_{+}\) and \(\varphi_{-}^{-1}\) act, and the parity is ultimately preserved.  Similarly, if \(\alpha\) changes the parity, exactly one of the \(\delta^{\prime}\) terms acts, so the sign is unchanged and exactly one of \(\varphi_{+}\) and \(\varphi_{-}^{-1}\) acts, again preserving the parity.

If \(\varphi_{+}\) acts, it takes \(k\) to the even number sharing its colour, and if \(\varphi_{-}^{-1}\) acts, it takes the integer from an even number to the odd number sharing its colour.  None of the other factors change the colour.
\end{proof}

\begin{proposition}
\label{proposition: first-order}
Associate with each colour \(c\in\left[C\right]\) an algebra of random matrices \(A_{c}\) with first-order limit distribution, and assume that the \(A_{c}\) are either independent or collectively possess a limit distribution.  For each \(c\in\left[C\right]\) let \(O_{c}\) be an independent Haar-distributed orthogonal matrix independent from any matrix in any of the \(A_{c}\).  Then the algebras \(O_{c}^{T}A_{c}O_{c}:=\left\{O_{c}^{T}XO_{c}:X\in A_{c}\right\}\) are free.
\end{proposition}
\begin{proof}
The \(r\)th cumulant of normalized traces (including the first) of a product of elements of the \(O_{c}^{T}A_{c}O_{c}\) may be expressed as \(N^{-r}\) times the value in Proposition~\ref{proposition: cumulants}.  The trace over \(K\left(\varphi,\delta^{\prime}\alpha\delta^{\prime}\right)^{-1}/2\) is a product of a number of traces, each of which is a classical random variable, and each associated to one colour \(c\in\left[C\right]\) according to Lemma~\ref{lemma: colours}.  Let us call them \(Z_{1},\ldots,Z_{r}\).  Then
\[\mathbb{E}\left(Z_{1},\ldots,Z_{r}\right)=\sum_{\rho\in{\cal P}\left(r\right)}k_{\rho}\left(Z_{1},\ldots,Z_{r}\right)\textrm{.}\]
Any cumulant of independent \(Z_{k}\) vanishes (if the \(A_{c}\) are independent), and any cumulant \(k_{s}\) for \(s>1\) vanishes asymptotically as a condition of the first-order limit distribution of the algebra associated to its colour (or the full algebra).  Thus the only remaining term is the one associated with the partition of singlets \(\rho=0_{r}\):
\[\lim_{N\rightarrow\infty}k_{1}\left(Z_{1}\right)\cdots k_{r}\left(Z_{r}\right)=\lim_{N\rightarrow\infty}\mathbb{E}\left(Z_{1}\right)\cdots\mathbb{E}\left(Z_{r}\right)\textrm{.}\]
Each expected value has a finite \(N\rightarrow\infty\) limit, again a condition of the first-order limit distribution.  So if \(r=1\), the \(N\rightarrow\infty\) limit exists, and if \(r>1\), it vanishes.

Let \(w:\left[n\right]\rightarrow\left[C\right]\) be an alternating word in the colours, and for \(1\leq k\leq n\), let \(X_{k}\in A_{w\left(k\right)}\) be a centred matrix (i.e. \(\mathbb{E}\left(\mathrm{tr}\left(X_{k}\right)\right)=0\)).  Then, letting \(w^{\prime}\left(k\right)=w\left(\lceil\frac{k}{2}\rceil\right)\), \(\varphi=\left(1,\ldots,2n\right)\), and \(\delta^{\prime}:k\mapsto\left(-1\right)^{k}k\), we have:
\begin{multline}
\mathbb{E}\left[\mathrm{tr}\left(O_{w\left(1\right)}^{T}X_{1}O_{w\left(1\right)}\cdots O_{w\left(n\right)}^{T}X_{n}O_{w\left(n\right)}\right)\right]\\=\sum_{\substack{\alpha=\alpha_{1}\cdots\alpha_{C}\\\alpha_{c}\in PM_{\mathrm{alt}}\left(\pm w^{\prime-1}\left(c\right)\right)}}N^{\chi\left(\varphi,\delta^{\prime}\alpha\delta^{\prime}\right)-2}\mathrm{wg}\left(\lambda\left(\alpha_{1}\right)\right)\cdots\mathrm{wg}\left(\lambda\left(\alpha_{C}\right)\right)\\\times\mathbb{E}\left(\mathrm{tr}_{K\left(\varphi,\delta^{\prime}\alpha\delta^{\prime}\right)^{-1}/2}\left(X_{1},I,\ldots,X_{n},I\right)\right)\textrm{.}
\label{formula: first-order}
\end{multline}
For a term surviving as \(N\rightarrow\infty\), we have \(\chi\left(\varphi,\delta^{\prime}\alpha\delta^{\prime}\right)=2\) and hence \(\left.\delta^{\prime}\alpha\delta^{\prime}\right|_{\left[2n\right]}\in S_{\mathrm{disc-nc}}\left(\varphi\right)\).  Being alternating each of its cycles must contain more than one element, so by Lemma~\ref{lemma: neighbours, disc}, there must be a \(k\in\left[2n\right]\) such that \(\delta^{\prime}\alpha^{-1}\delta^{\prime}\left(k\right)=\varphi\left(k\right)\).  For even \(k\neq 2n\), \(w^{\prime}\left(\varphi\left(k\right)\right)\neq w^{\prime}\left(k\right)\), so \(\delta^{\prime}\alpha^{-1}\delta^{\prime}\) cannot connect \(k\) and \(\varphi\left(k\right)\).  Thus \(k\) is odd (or \(k=2n\); if so, we choose the other \(k\)).  Then \(K\left(\varphi,\delta^{\prime}\alpha\delta^{\prime}\right)\left(k\right)=k\), so cycle \(\left(k\right)\) appears in its cycle decomposition, so \(Z_{s}=\mathrm{tr}\left(X_{\frac{k+1}{2}}\right)\) for some \(s\).  Since \(\mathbb{E}\left[\mathrm{tr}\left(X_{\frac{k+1}{2}}\right)\right]=0\), the term in (\ref{formula: first-order}) associated to this \(\alpha\) must vanish.  Thus (\ref{formula: first-order}) vanishes asymptotically, proving the result.
\end{proof}

\begin{proposition}
\label{proposition: limit distribution}
Associate with each colour \(c\in\left[C\right]\) an independent algebra of random matrices \(A_{c}\) with second-order limit distribution.  For each \(c\in\left[C\right]\) let \(O_{c}\) be an independent Haar-distributed orthogonal matrix independent from any matrix in any of the \(A_{c}\).  Then the algebra generated by the \(O_{c}^{T}A_{c}O_{c}\) has a second-order limit distribution.
\end{proposition}
\begin{proof}
Expanding a cumulant of terms in the algebras generated by the \(O_{c}^{T}A_{c}O_{c}\) in the form of Proposition~\ref{proposition: cumulants}, by Lemma~\ref{lemma: colours}, each trace appearing in this expansion is associated with one colour \(c\in\left[C\right]\), so any cumulant \(k_{\tau}\) in which traces of different colours appear in the same block of \(\tau\) vanishes.  The remaining cumulants must satisfy the convergence requirements as cumulants of an algebra with a second-order limit distribution.
\end{proof}

\begin{theorem}
Associate with each colour \(c\in\left[C\right]\) an algebra of random matrices \(A_{c}\) with second-order limit distribution, and assume that either the \(A_{c}\) are independent, or the algebra generated by the \(A_{c}\) has a second-order limit distribution.  For each \(c\in\left[C\right]\) let \(O_{c}\) be an independent Haar-distributed orthogonal matrix independent from any matrix in any of the \(A_{c}\).  Then the algebras \(O_{c}^{T}A_{c}O_{c}:=\left\{O_{c}^{T}XO_{c}:X\in A_{c}\right\}\) are free of real second-order.
\end{theorem}
\begin{proof}
We have first-order freeness by Proposition~\ref{proposition: first-order} and the existence of a second-order limit distribution by Proposition~\ref{proposition: limit distribution} (if the \(A_{c}\) are independent; if they collectively have a second-order limit distribution then by Proposition~\ref{proposition: cumulants}, so do the \(O_{c}^{T}A_{c}O_{c}\)).

Let \(v:\left[p\right]\rightarrow\left[C\right]\) and \(w:\left[q\right]\rightarrow\left[C\right]\) be cyclically alternating words in the colours, and let \(X_{1},\ldots,X_{p}\) and \(Y_{1},\ldots,Y_{q}\) be centred random matrices with \(X_{k}\in A_{v\left(k\right)}\), \(k\in\left[p\right]\), and \(Y_{k}\in A_{w\left(k\right)}\), \(k\in\left[q\right]\).  We let
\[w^{\prime}\left(k\right)=\left\{\begin{array}{ll}v\left(\lceil\frac{k}{2}\rceil\right)&k\in\left[2p\right]\\w\left(\lceil\frac{k}{2}-p\rceil\right)&k\in\left[2p+1,2q\right]\end{array}\right.\textrm{.}\]
Let
\[\varphi\left(1,\ldots,2p\right)\left(2p+1,\ldots,2p+2q\right)\textrm{,}\]
\[\varphi_{\mathrm{ext}}=\left(1,\ldots,2p\right),\varphi_{\mathrm{int}}=\left(2p+1,\ldots,2+2q\right)\textrm{,}\]
\[\varphi_{\mathrm{op}}=\left(1,\ldots,2p\right)\left(-2p-2q,\ldots,-2p-1\right)\textrm{.}\]
Let \(\delta^{\prime}:k\mapsto\left(-1\right)^{k}k\).  We have, in the notation of Proposition~\ref{proposition: cumulants}:
\begin{multline}
\label{formula: second-order expansion}
k_{2}\left(\mathrm{Tr}\left(O_{v\left(1\right)}^{T}X_{1}O_{v\left(1\right)}\cdots O_{v\left(p\right)}^{T}X_{p}O_{v\left(p\right)}\right),\right.\\\left.\mathrm{Tr}\left(O_{w\left(1\right)}^{T}Y_{1}O_{w\left(1\right)}\cdots O_{w\left(q\right)}^{T}Y_{q}O_{w\left(q\right)}\right)\right)
\\=\sum_{\substack{\alpha=\alpha_{1}\cdots\alpha_{C}\\\alpha_{c}\in PM_{\mathrm{alt}}\left(\pm w^{\prime-1}\left(c\right)\right)}}N^{\chi\left(\varphi,\delta^{\prime}\alpha\delta^{\prime}\right)-2}\sum_{\substack{\rho\in{\cal P}\left(n\right):\pi\preceq\rho\preceq\ker\left(w^{\prime}\right)\\\tau\in{\cal P}\left(s\right)\\\varphi\vee\rho\vee\tau_{\sigma}=1_{2p+2q}}}C_{\pi,\pi,\rho}\\\times k_{\tau}\left(\mathrm{tr}_{\sigma^{-1}_{1}}\left(X_{1},\ldots,I,X_{p},I,Y_{1},I,\ldots,Y_{q},I\right),\ldots,\right.\\\left.\mathrm{tr}_{\sigma^{-1}_{s}}\left(X_{1},\ldots,I,X_{p},I,Y_{1},I,\ldots,Y_{q},I\right)\right)\textrm{.}
\end{multline}
We consider the terms that survive as \(N\rightarrow\infty\).  If \(\alpha\) does not connect \(\pm\left[2p\right]\) and \(\pm\left[2p+1,2q\right]\), then at least one of \(\rho\) or \(\tau_{\sigma}\) must, in which case that partition must have at least one block of size \(2\).  In either case, this block corresponds to a factor which is \(O\left(N^{-2}\right)\), so in a surviving term, this is the only block of size greater than one in either partition.  Furthermore, we must have \(\chi\left(\varphi,\delta^{\prime}\alpha\delta^{\prime}\right)=4\) (so \(\delta^{\prime}\alpha\delta^{\prime}\) does not connect \(\left[2p\right]\) to \(-\left[2p\right]\) or \(\left[2p+1,2p+2q\right]\) to \(-\left[2p+1,2p+2q\right]\), with \(\left.\delta^{\prime}\alpha\delta^{\prime}\right|_{\left[2p\right]}\in S_{\mathrm{disc-nc}}\left(\varphi_{\mathrm{ext}}\right)\) and \(\left.\delta^{\prime}\alpha\delta^{\prime}\right|_{\left[2p+1,2p+2q\right]}\in S_{\mathrm{disc-nc}}\left(\varphi_{\mathrm{int}}\right)\)).  By Lemma~\ref{lemma: neighbours, disc}, we can then find two \(k\) in \(\left[2p\right]\) and two in \(\left[2p+1,2p+2q\right]\) such that \(\delta^{\prime}\alpha^{-1}\delta^{\prime}\left(k\right)=\varphi\left(k\right)\), all of which are odd (since for \(k\) even, \(w^{\prime}\left(\varphi\left(k\right)\right)\neq w^{\prime}\left(k\right)\), so they cannot be in the same cycle of \(\delta^{\prime}\alpha^{-1}\delta^{\prime}\)).  For each, \(K\left(\varphi,\delta^{\prime}\alpha\delta^{\prime}\right)^{-1}\left(k\right)=k\), so the cycle \(\left(k\right)\) appears in \(K\left(\varphi,\delta^{\prime}\alpha\delta^{\prime}\right)^{-1}\).  In addition, since in a surviving term, \(\tau_{\sigma}\) can have a block connecting at most two cycles of \(K\left(\varphi,\delta^{\prime}\pi\delta^{\prime}\right)^{-1}/2\), we must have at least one cycle \(\left(k\right)\) in its own block of \(\tau_{\sigma}\).  The term then contains the expected value of the trace of a centred matrix, so it vanishes.

Otherwise, \(\alpha\) connects \(\pm\left[2p\right]\) and \(\pm\left[2p+1,2q\right]\).  Highest order terms have \(\chi\left(\varphi,\delta^{\prime}\alpha\delta^{\prime}\right)=2\) (so by Theorem~\ref{theorem: unoriented noncrossing}, we have either \(\left.\delta^{\prime}\alpha\delta^{\prime}\right|_{\left[2p+2q\right]}\in S_{\mathrm{ann-nc}}\left(\varphi\right)\) or \(\left.\delta^{\prime}\alpha\delta^{\prime}\right|_{\varphi_{\mathrm{op}}}\in S_{\mathrm{ann-nc}}\left(\varphi_{\mathrm{op}}\right)\)) and have both \(\rho\) and \(\tau\) as small as possible.

Since each cycle of \(\alpha\) must contain more than one element, any surviving \(\alpha\) must be a pairing where each pair consists of an element from \(\pm\left[2p\right]\) and an element from \(\pm\left[2p+1,2p+2q\right]\).  If \(p\neq q\), there are no such pairings, so the covariance vanishes asymptotically as desired.  If \(p=q\), we will consider two cases: \(\alpha\left(2p\right)\) is odd, and \(\alpha\left(2p\right)\) is even.

If \(\alpha\left(2p\right)\) is odd, define \(k\) by \(\alpha\left(2p\right)=2k-1\).  Then the cycle \(\left(2p,-2k+1\right)\) appears in \(\alpha\), and hence the cycle \(\left(2p,2k-1\right)\) appears in \(\delta^{\prime}\alpha\delta^{\prime}\).

We show by induction that for \(i\in\left[2p\right]\), \(\delta^{\prime}\alpha\delta^{\prime}\left(i\right)=2k-i-1\) (taken in \(\left[2p+1,2p+2q\right]\), modulo \(2q\)).  Let \(a\) be the first integer in \(\left[2p\right]\) for which this does not hold.  Let \(b=2k-a-1\), let \(c=\delta^{\prime}\alpha\delta^{\prime}\left(a\right)\), and let \(d=\delta^{\prime}\alpha\delta^{\prime}\left(b\right)\).  We know that \(\left.\lambda_{2p,2k-1}\right|_{\left\{a,b,c,d\right\}}\left(b\right)=a\), since we know the partners of every element between (so \(c\) and \(d\) cannot be any of these elements).  We also know that \(d\in\left[2p\right]\) and \(c\in\left[2p+1,2p+2q\right]\), so \(\left.\lambda_{2p,2k-1}\right|_{\left\{a,b,c,d\right\}}=\left(a,d,c,b\right)\), which is the third annular-crossing condition.

Such a \(\alpha\) appears in (\ref{formula: second-order expansion}) only if \(w^{\prime}\left(i\right)=w^{\prime}\left(2k-i-2\right)\) for all \(i\in\left[2p\right]\), that is, if \(v\left(i\right)=w\left(k-i\right)\) for all \(i\in\left[p\right]\).  If so, we calculate that the cycles of \(K\left(\varphi,\delta^{\prime}\alpha\delta^{\prime}\right)^{-1}/2\) are \(\left(i,2k-i-2\right)\), \(1\leq i\leq 2p\), and the contribution of the term is:
\[\lim_{N\rightarrow\infty}\prod_{i=1}^{p}\mathbb{E}\left(X_{i}Y_{k-i}\right)\\=\lim_{N\rightarrow\infty}\prod\mathbb{E}\left(O_{v\left(i\right)}^{T}X_{i}O_{v\left(i\right)}O_{w\left(k-i\right)}^{T}Y_{k-i}O_{w\left(k-i\right)}\right)\textrm{.}\]
If there is an \(i\) such that \(v\left(i\right)\neq w\left(k-i\right)\), then the term \(\alpha\) does not appear in the sum, so we make take the contribution to be zero.  By Proposition~\ref{proposition: first-order},
\[\lim_{N\rightarrow\infty}\mathbb{E}\left(\mathrm{tr}\left(O_{v\left(i\right)}^{T}X_{i}O_{v\left(i\right)}O_{w\left(k-i\right)}^{T}Y_{k-i}O_{w\left(k-i\right)}\right)\right)=0\textrm{,}\]
which is again the desired contribution.

The case \(\alpha\left(2p\right)\) even is similar.  Define \(k\) by \(\alpha\left(2p\right)=2k\).  Then the cycle \(\left(2p,-2k\right)\) appears in \(\delta^{\prime}\alpha\delta^{\prime}\), connecting \(\left[2p\right]\) to \(-\left[2p+1,2p+2q\right]\), so \(\left.\delta^{\prime}\alpha\delta^{\prime}\right|_{\varphi_{\mathrm{op}}}\in S_{\mathrm{ann-nc}}\left(\varphi_{\mathrm{op}}\right)\).  By induction, \(\delta^{\prime}\alpha\delta^{\prime}\left(i\right)=-2k-i\).  There is nonzero contribution only if \(v\left(i\right)=w\left(k+i\right)\) for all \(i\in\left[p\right]\), and as above, in either case the contribution of the term is
\[\lim_{N\rightarrow\infty}\prod_{i=1}^{p}\mathbb{E}\left(O_{v\left(i\right)}^{T}X_{i}O_{v\left(i\right)}O_{w\left(k+i\right)}^{T}Y_{k+i}^{T}O_{w\left(k+i\right)}\right)\textrm{.}\]  The result follows.
\end{proof}

We note that conjugating by any orthogonal matrix does not change the value of traces, so we may also have one ensemble which is not conjugated.

Several corollaries follow:

\begin{corollary}
Any combination independent matrices drawn from orthogonally invariant distributions (including real Ginibre matrices, Gaussian orthogonal ensemble matrices, real Wishart matrices, and Haar-distributed orthogonal matrices) and one other ensemble (possibly constant matrices) are real second-order free.
\end{corollary}

\begin{corollary}
An algebra of random matrices \(A\) with second-order limit distribution is free of real second-order from \(O^{T}AO\), where \(O\) is a Haar-distributed orthogonal matrix independent from \(A\).
\end{corollary}

\bibliography{paper}
\bibliographystyle{plain}

\end{document}